\documentclass{article}

\usepackage{hyperref}
\hypersetup{
 colorlinks = true,
 linkcolor = blue,
 filecolor = blue, 
 urlcolor = blue,
 citecolor = blue,
}
\usepackage{graphicx}

\usepackage[all]{xy}

\usepackage{amsmath}
\usepackage{amssymb}
\usepackage{enumitem}
\usepackage{graphicx}
\usepackage{mathdots}
\usepackage{tikz, tikz-cd, tkz-graph}\usepackage{color}
\usepackage{array}
\newcolumntype{M}[1]{>{\centering\arraybackslash}m{#1}}

\usepackage[margin=4cm]{geometry}

\usepackage{amsthm}
\newtheorem{defn}{Definition}
\newtheorem{thm}[defn]{Theorem}
\newtheorem{pro}[defn]{Proposition}
\newtheorem{cor}[defn]{Corollary}
\newtheorem{lem}[defn]{Lemma}

\theoremstyle{remark} % 'style changed again'
\newtheorem{exa}[defn]{Example}
\newtheorem{rem}[defn]{Remark}

\newcommand{\CC}{\mathbb C}

\newcommand{\FF}{\mathbb F}
\newcommand{\HH}{\mathbb H}

\newcommand{\LL}{\mathbb L}
\newcommand{\NN}{\mathbb N}
\newcommand{\PP}{\mathbb P}
\newcommand{\GG}{\mathbb G}
\newcommand{\QQ}{\mathbb Q}

\newcommand{\VV}{\mathbb V}
\newcommand{\ZZ}{\mathbb Z}

\newcommand{\Ocal}{\mathcal O}

\newcommand{\Hcal}{\mathcal H}
\newcommand{\Hbar}{\mathbb H}

\newcommand{\diag}{\mathrm {diag}}

\newcommand{\id}{\operatorname{id}}
\newcommand{\Aut}{\operatorname{Aut}}
\newcommand{\GL}{\operatorname{GL}}
\newcommand{\Jac}{\operatorname{Jac}}

\newcommand{\SL}{\operatorname{SL}}

\newcommand{\im}{\operatorname{im}}

\newcommand{\age}{\operatorname{age}}

\newcommand{\Hom}{\operatorname{Hom}}

\newcommand{\Spec}{\operatorname{Spec}}

\newcommand{\al}{\alpha}
\newcommand{\mf}{\mathfrak}

\newcommand{\fie}{\varphi}

\newcommand{\la}{\lambda}

\newcommand{\cxi}{\mathtt{i}}

\def\pmmu{{\pmb \mu}}

\def\CR{\operatorname{CR}}

\def\ol{\overline}
\def\ul{\underline}
\def\wt{\widetilde}

\newcommand{\smallvee}{\mathrel{\text{\raisebox{0.25ex}{\scalebox{0.6}{$\vee$}}}}}
\makeatletter
\let\@fnsymbol\@arabic
\makeatother

\title{Mirror symmetry and automorphisms}
\author{Alessandro Chiodo\thanks{Supported by the ANR project ``Categorification in Algebraic Geometry'', CANR-17-CE40-0014, the ANR project 
``Enumerative Geometry'', PRC ENUMGEOM.}, Elana Kalashnikov\thanks{Supported by the Engineering and Physical Sciences Research Council [EP/L015234/1], the EPSRC Centre for Doctoral Training in Geometry and Number Theory (The London School of Geometry and Number Theory), University College London.}}

\begin{document}

\maketitle

\vspace{-20pt}

\begin{abstract} 
\noindent 
We show that there is an extra dimension to the mirror duality
discovered in the early nineties by Greene--Plesser and Berglund--H\"ubsch. Their duality matches cohomology classes  of  
two  Calabi--Yau orbifolds. When both orbifolds are equipped with an automorphism $s$ of the same order, our mirror duality involves the weight of the action of $s^*$ 
on cohomology. In particular it matches  the respective $s$-fixed loci,
which are not Calabi--Yau in general.
When applied to K3 surfaces with non-symplectic automorphism $s$ of odd prime order, this 
provides a proof that Berglund--H\"ubsch mirror symmetry implies K3 lattice mirror symmetry replacing earlier case-by-case treatments.
\end{abstract}

\setcounter{tocdepth}{1} 
%\tableofcontents

%%%%%%%%%%%%%%%%%%%%%%%%%%%%%%%%%%%%%%%%%%%%%%%%%%%%%%%%%%%%%%%%%%%%%%%%%%%

%\setcounter{section}

\section{Introduction} 
The earliest formulation of mirror symmetry relates pairs of $d$-dimensional Calabi--Yau manifolds $X, X^\vee$ with {mirror Hodge diamonds}: 
\[h^{p,q}(X)=h^{d-p,q}(X^\vee).\]
In the early 1990s, physicists Greene, Morrison, and Plesser found many such mirror pairs \cite{greene}, starting with a Calabi--Yau (and Fermat) hypersurface in projective space and constructing a mirror, which is a resolution of the quotient of the same hypersurface by a finite group. In 1992, this construction was generalized by Berglund--H\"ubsch \cite{BH}, starting with a Calabi--Yau given as a quotient of a more general hypersurface in weighted projective spaces by a finite group. The hypersurface is a Calabi--Yau orbifold defined as the zero locus of a quasi-homogenous polynomial $W=\sum_{i=0}^n \prod_{j=0}^n x_{j}^{m_{ij}}$ such that $W$ is {non-degenerate} and ``{invertible}'' (\emph{i.e.} with as many variables as monomials). After quotienting out by a finite group $H$ of diagonal symmetries within $\SL(n+1;\CC)$ one obtains 
the orbifold $\Sigma_{W,H}$.
 The mirror $\Sigma_{W^\vee,H^\vee}$ is another such quotient of a hypersurface modulo a finite group.  The  hypersurface is given by the polynomial $W^\vee$, defined by transposing the matrix of the exponents $E=[m_{ij}]$ of $W$. The group $H^\vee$ is a subgroup of $\SL(n+1;\CC)$ Cartier dual to $H$ and 
preserving $W^\vee$, see \eqref{eq:groupduality}. Then, the mirror 
duality can be stated in terms of orbifold Chen--Ruan cohomology as 
\begin{equation}\label{eq:MScohom} H_{\CR}^{p,q}(\Sigma_{W,H};\CC)=H_{\CR}^{d-p,q}(\Sigma_{W^\vee,H^\vee};\CC),
\end{equation}
which implies the same relation in ordinary cohomology whenever there exists crepant resolutions.
% 
%$\pi\colon \wt \Sigma\to \Sigma$ with zero discrepancy (``crepant''): 
%$\pi^*\omega_\Sigma=\omega_{\wt\Sigma}$.
%

\medskip

The striking mirror relation above becomes elementary when 
we look at it through the lenses of singularity theory or, in 
physics terminology, the Landau--Ginzburg (LG) model. 
This happens because mirror symmetry holds for LG models without any 
Calabi--Yau condition. 
In this paper we present this change of perspective through the LG model via 
the crepant resolution of a singularity, see Section \ref{LG}. 
This not only allows us to simplify previous proofs of 
LG/CY correspondence by the first author 
with Ruan \cite{ChRu}; it also yields a new statement 
of mirror symmetry 
%mixing the 
%Hodge decomposition and the dimension of the eigenspaces 
%of an automorphism operating
%on both sides of the mirror. 
%More precisely, the mirror isomorphism 
relating the fixed loci of powers of 
an isomorphism $s$ of $\Sigma$, the Hodge decomposition, and the 
weights the representation $s^*$ in cohomology. 
%In the presence of such an automorphism, our result can be seen as both refining and generalizing the original mirror statement. 

Let $W=x_0^k+f(x_1,\dots,x_n)$ be a non-degenerate,  quasi-homogenous, invertible polynomial. 
Let us consider again the automorphisms groups $H\subseteq \Aut W$ and its dual $H^\vee\in \Aut W$ within
$\SL(n+1;\CC)$. The 
Calabi--Yau orbifolds $\Sigma_{W,H}$, $\Sigma_{W^\vee,H^\vee}$ are equipped with the 
action by the group $\pmmu_k$ of $k$th roots of unity spanned by 
$s\colon x_0 \mapsto e^{2\pi i/k} x_0$. For $i$ in 
the group of characters $\ZZ/k=\Hom(\pmmu_k;\GG_m)$ we 
consider the weight-$i$ term of cohomology 
$$H^*(\ \ ,\CC)_{i}=\{x\mid s^*x=i(s)x\}.$$
The first statement is that 
the $s$-invariant cohomology mirrors the ``moving'' cohomology: the sum of 
all cycles of nonvanishing weight.

\medskip
\noindent{\textbf{Theorem A}} (see Thm \ref{thm:MS_CY}, part 1). 
\textit{Consider the mirror pair $s\colon \Sigma_{W,H}\to \Sigma_{W,H}$ 
and $s\colon \Sigma_{W^\vee,H^\vee}\to \Sigma_{W^\vee,H^\vee}$. We have}
\[H^{p,q}_{orb}(\Sigma_{W,H};\CC)_{0}\cong 
\bigoplus_{i=1}^{k-1} H^{d-p,q}_{orb}(\Sigma_{W^\vee,H^\vee};\CC)_{i},\]
\noindent\textit{where $d=n-1$ is the dimension of $\Sigma_{W,H}$.}
\medskip

The locus of  geometric points of  $\Sigma_{W,H}$ which are fixed by $s$ 
also exhibits a mirror phenomenon. Since $\Sigma_{W,H}$ is a stack, let us provide a definition for this $s$-fixed locus. 
For $s$ a finite order automorphism acting on a smooth Deligne--Mumford orbifold,
we consider the graph of $\Gamma_s\colon \mathfrak X\to \mathfrak X\times \mathfrak X$ 
and its intersection with the graph of the identity (the diagonal morphism)
\[\mathfrak X{\underset{s, \ \mathfrak X\times \mathfrak X, \ \id}{\times}}\mathfrak X,\]
(we write $s$ and $\id$ instead of the respective graphs).  
We recall that orbifold cohomology is simply the (age-shifted) cohomology of this product for
$s=\id$. The $s$-orbifold cohomology is defined as the age-shifted cohomology of the above 
fibred product in general (see Defn.~\ref{defn:g_orbifoldcoh}). 
This is a bi-graded vector space and, if the coarse space $X$ of $\mf X$ admits a crepant resolution $\wt X$ where $s$ lifts, there is a bidegree-preserving isomorphism
$H_s^*(\mf X;\CC)\cong H_s^*(\wt X;\CC),$
where the right hand side is the age-shifted cohomology of the $s$-fixed locus in $\wt X$, see Prop.~\ref{pro:crc}.

%Suppose $W=x_0^k+f(x_1,\dots,x_n)$, $H\subseteq \Aut W$, and $s\colon \Sigma_{W,H}\to 
%\Sigma_{W,H}$ as above. 
Under the same conditions on $W$ and $H$ as above, set $\Sigma=\Sigma_{W,H}$ and $\Sigma^\vee=\Sigma_{W^\vee,H^\vee}$. If the order $k$ of $s$ is not prime, then $s$ acts non-trivially on the fixed locus of powers of $s$. The $s$-moving cohomology of the fixed locus of powers of $s$ mirrors the same on $\Sigma^\vee$, interweaving the weight and the exponent of the power of $s$. 

\medskip 
\noindent \textbf{Theorem B} (see Thm \ref{thm:MS_CY}, part 3).
\textit{Let $b,t\neq 0$. Then, we have }
\[\textstyle{H^{p,q}_{s^b}(\Sigma)\left(\frac{b}k\right)_{t}\cong 
H^{d-p,q}_{s^{-t}}(\Sigma^\vee)\left(\frac{k-t}k\right)_{-b}},\]
\noindent \textit{where $d=n-2$, 
the largest dimension of the components of the $s$-fixed locus.}
\smallskip

Finally, also the fixed cohomology of each power $s^j$  exhibits a mirror phenomenon, but only after adding certain moving cycles in $\Sigma$. Namely, the cycles we add are
all those whose weight differs from  $0$ (\emph{i.e.}~moving cycles) and from $j$ (the exponent of $s$). 
We denote this group 
by $\overline{H}^{p,q}_{\id,j}(\Sigma)$, see \eqref{eq:movingcycles}. 

\medskip 
\noindent \textbf{Theorem C} (see Thm \ref{thm:MS_CY}, part 2).
\textit{
For $0<j<k$, we have}
\begin{equation*}
\left[H^{p,q}_{s^j}(\Sigma)\left(\tfrac{j}{k}\right)\right]^s \oplus \overline{H}^{p,q}_{\id,j}(\Sigma)\cong \left[H^{d-p,q}_{s^j}(\Sigma^\vee)\left(\tfrac{j}{k}\right)\right]^s \oplus \overline{H}_{\id,j}^{d-p,q}(\Sigma^\vee).\end{equation*}
The correcting terms $\overline{H}^{*}$ disappear 
when $k=2$ (for $k=2$, we have $s^j=s$ and there 
is no positive weight except $1$). This shows how the statement above specialises 
to the construction of Borcea--Voisin mirror pairs (see \cite{CKV}).
%Indeed one of the initial  motivations of this work is to better understand Borcea--Voisin mirror symmetry and recast it in the same context as other occurrences of mirror symmetry.  

In dimension 2, and after resolving, these results are about mirror symmetry for K3 surfaces with non-symplectic automorphisms. Suppose $X$ and $X^\vee$ are crepant resolutions of $\Sigma_{W,H}$ and $\Sigma_{W^\vee,H^\vee}$ where $W$ is a polynomial in $4$ variables. The above mirror theorems imply that the topological invariants of the fixed locus of the K3 surface $X$ controls that of $X^\vee$; we refer to Corollary \ref{cor:primeinv} for simple formulae on the number of fixed points and the genera of the fixed curves. 
The automorphism $s$ also gives the K3 surface a \emph{lattice polarization}: $H^2(X,\ZZ)^s$. There is another version of mirror symmetry for lattice polarized K3 surfaces, arising from the work of Nikulin \cite{Ni1}, Dolgachev \cite{Do-mirror}, Voisin \cite{Vo}, and Borcea \cite{Borcea}. When the order of $s$ is odd and prime, this lattice is characterised by the invariants $(r,a)$: the rank and the discriminant. Families of lattice polarized K3 surfaces come in mirror pairs, and in the odd prime case this mirror symmetry takes a lattice with invariants $(r,a)$ to $(20-r,a)$. 
The following corollary is a theorem of Comparin, Lyons, Priddis, and Suggs \cite{CLPS} proven by case-by-case analysis. Here, it is shown directly from 
the above statements (see Thm.~\ref{thm:LPK3}).

\smallskip 
\noindent \textbf{Corollary (\cite{CLPS}).} 
\emph{Let $p$ be prime and different from $2$. Let $\Sigma_{W,H}$ and $\Sigma_{W^\vee,H^\vee}$ be mirror K3 orbifolds with order-$p$ automorphisms $s, s^\vee$, and let $\Sigma$ and $\Sigma^\vee$ be crepant resolutions with automorphisms also denoted $s, s^\vee$. Then  $\Sigma$ and $\Sigma^\vee$ are mirror as lattice polarized K3 surfaces. }
\smallskip

\subsection{Relation to previous work}
This paper generalises the results of \cite{CKV}. There, only involutions were considered;  here the mirror theorems apply to automorphism of any order. There, Theorems A and C are simpler (invariant classes mirror anti-invariant classes in Theorem A and no extra terms appear in Theorem C). Theorem B does not apply in the involution case. In the above Corollary, we do not consider the order-2 case treated in \cite{CKV}; in the present paper this allows us to deduce the lattice mirror symmetry statement of \cite{CLPS} in full. 

Section \ref{LG} restates and recasts the proof of 
 mirror symmetry through LG models and the correspondence
between cohomology and LG models in terms of resolutions of singularities (see Theorem \ref{thm:MSCY}). 
This may be regarded as the outcome of the work of many authors, we refer to 
\cite{Krawitz},  \cite{Borisov},  \cite{kaufmann}
 \cite{ChRu}, \cite{ebeling_takahashi_2011}, \cite{GZE_Saito} and \cite{Gusein-Zade-Ebeling} and \cite{ChiodoNagel} 
 validating over the years the  
 approach of the physicists Intriligator--Vafa \cite{vafa} and Witten \cite{Witten}.
It is also worth mentioning that the main object of our study, a polynomial $W=x_0^k + f(x_1,...,x_n)$ with the cyclic symmetry group of $k$th roots of unity acting on $x_0$, was used in Varchenko's proof of semicontinuity of 
Steenbrink's spectra of singularities (\cite{Va} and \cite{Steenbrink2}). 
We hope that this may lead to further
explanations of mirror symmetry in the framework of singularity theory. 
In particular, our setup only concerns hypersurfaces in weighted projective space, it would be interesting to see if it extends to other contexts 
where mirror constructions are known. 

Finally, it is worth mentioning that the work of Comparin, Lyons, Priddis, and Suggs \cite{CLPS}, earlier work of Artebani, Boissi\`ere and Sarti \cite{ABS} and more generally Nikulin's classification \cite{Ni1} yield 
several tables summarising 
explicit treatments of K3 surfaces via resolution of singularities. 
Much of these data are now 
embodied into the $s$-weighted Hodge numbers of Theorems A, B, and C. We provide some 
examples for this in the tables at the end of \S \ref{sect:geometry}.

\subsection{Structure of the paper} 
Section \ref{sect:terminology} states notation and terminology. 
Section \ref{setup} presents the Berglund--H\"ubsch mirror symmetry construction.
Section \ref{inertia}  sets up our generalisation of 
orbifold cohomology sensing the $s$-fixed locus: $s$-orbifold cohomology.
Section \ref{LG} illustrates and reproves the transition to Landau-Ginzburg models which is crucial in the proof. In particular it  provides a  straightforward description of the LG/CY correspondence from 
the crepant resolution conjecture without using the combinatorial model of \cite{ChRu}.
Section \ref{sect:unprojected} is the technical heart of the paper; it proves 
the main theorem (Theorem \ref{thm:MS_LG}) on the LG side. 
Section \ref{sect:geometry} translates the result from the 
LG side  to the CY side. It contains Theorem \ref{thm:MS_CY}
proving the statements A, B, and C  and 
Theorem \ref{thm:LPK3} specialising to K3 surfaces.

\subsection*{Aknowledgements}
We are grateful to Behrang Noohi, whose explanations 
clarified inertia stacks to us. 
We are grateful to Baohua Fu, Lie Fu, Dan Israel, and Takehiko Yasuda for many 
helpful conversations. We thank Davide Cesare Veniani, with
whom we started studying involutions of Calabi--Yau orbifolds, for continuing to share his insights and
expertise. 

\section{Terminology} \label{sect:terminology}
Deligne--Mumford orbifolds are smooth separated Deligne--Mumford stacks
with a dense open subset isomorphic to an algebraic variety. 
\subsection{Conventions}
We work with schemes and stacks over the complex numbers. All schemes are Noetherian and separated.
By linear algebraic group we mean a closed subgroup of 
$\GL_m(\CC)$ for some $m$.
We often need to identify 
a stack locally. 
In order to avoid repeated mention of 
\'etale localization or strict Henselizations, we often use the expression 
``the 
local picture of the stack $\mf X$ at the geometric point $x\in \mf X$ is the same as $\mf U$ at $u\in \mf U$''. 
By this we mean that the strict Henselization of $\mf X$ at $x$ is the same of that 
of $\mf U$ at $u$. 
Often it is enough to say that there is an \'etale neighbourhood $X'$ of $x$ 
 and an isomorphism $X'\to U'$ with an \'etale neighbourhood $U'$ of $u$.
We refer to \cite[54.33.2]{stacksproject} for a definition of the 
strict Henselization 
and to \cite[\S1.2,5]{ACV} for further discussion (see in particular the ``algebra-to-analysis translation'', where strict Henselizations 
are described analytically as the germ of $X$ at $x$).

\subsection{Notation} We list here notation that occurs throughout the entire paper. 
%\vspace{7pt} 
%\noindent 
%{\bf Notation}

\vspace{5pt} 

\noindent\begin{tabular}{ll}
$V^K$& is the invariant subspace of a vector space $V$ linearized by a finite group $K$.\\
$\PP(\pmb w)$ & is the quotient stack $[(\CC^n\setminus \pmb 0)/\GG_m]$, where $\GG_m$ acts with weights $\pmb w$.\\
$Z(f)$ & is the variety defined as 
 zero locus of $f\in \CC[x_1,\dots, x_n]$.\\
$H(a,b)$& is the bigraded vector space with shifted grading: $[H(a,b)]^{p,q}=H^{p+a,q+b}$.\\
 
\end{tabular}

\medskip

\begin{rem}[zero loci]\label{rem:zeroloci} We add the subscript 
$\PP(\pmb w)$ when we refer to the zero locus in $\PP(\pmb w)$ of a polynomial $f$ which is $\pmb w$-weighted homogeneous. In this way we have
\[Z_{\PP(\pmb w)}(f)=[U/\GG_m],\qquad \text{\ \ with $U=Z(f)\subset \CC^n\setminus \pmb 0$.}\]
\end{rem}

\begin{rem}[degree shift]
We often write $H(a)$ for $H(a,a)$.
\end{rem} 

\begin{rem}[cohomology coefficients]
We only consider cohomology with $\CC$ coefficients; therefore, we sometimes write $H^*(X;\CC)$ as $H^*(X).$
\end{rem}

\begin{rem}[graphs and maps]\label{rem:abuse_notn_graph}
Given an automorphism $\alpha$ of $\mathfrak X$, 
we write $\Gamma_\alpha$ for 
the graph $\mathfrak X\to 
\mathfrak X\times \mathfrak X$. However, to simplify 
formul\ae, 
we often abuse notation 
and use $\alpha$ for the graph $\Gamma_\al$ as well as the automorphism. In this way, in subscripts, the diagonal 
$\Delta \colon \mathfrak X\to \mathfrak X\times \mathfrak X$ will be often written as $\id_{\mathfrak X}$ or simply $\id$.
\end{rem}

\section{Setup}\label{setup}
We recall the general setup of non-degenerate polynomials $P$ where
the theory of Jacobi rings applies. Then we introduce polynomials 
of the special form $$W(x_0,x_1,\dots,x_n)=x_0^k+f(x_1,\dots, x_n)$$
for $n>0$.

\subsection{Non-degenerate polynomials}\label{sect:nondegpoly}
We consider quasi-homogeneous polynomials $P$ of degree $d$ and 
of weights $w_0,\dots,w_n$
$$P(\la^{w_0}x_0,\dots,\la^{w_n}x_n)=\la^dP(x_0,\dots,x_n),$$
for all $\la\in \CC$.
We assume that that the polynomial $P$ is non-degenerate; \emph{i.e.}
the choice of weights and 
degree is unique and  the partial derivatives of $W$ vanish 
simultaneously only at the origin. 
We consider the zero locus 
$$\Sigma_P=Z_{\PP(\pmb w)}(P)\subset \PP(\pmb w)$$ 
which is, by non-degeneracy, a smooth hypersurface within 
the weighted projective stack $\PP(\pmb w)=
[(\CC^{n+1}\setminus \pmb 0)/\GG_m]$ with $\GG_m$ acting 
with weights  $w_0,\dots,w_n$.
The polynomial is of Calabi-Yau type if 
\begin{equation}\label{eq:CY}\sum_{i=0}^n{w_i}=d.\end{equation}
This implies that the canonical bundle of ${\Sigma_P}$
is trivial; we refer to $\Sigma_P$ as a 
Calabi--Yau orbifold.

Because $P$ is non-degenerate, 
the group of its diagonal automorphisms
$$\Aut_P=\{\diag(\al_0,\dots,\al_n)\mid P(\al_0x_0,\dots, \al_nx_n)=P(x_0,\dots,x_n)\}$$
 is finite. Indeed, the $n\times h$ matrix $E=(m_{i,j})$ of the exponents of 
$P=\sum_{i=1}^h c_i \prod_{j=0}^n x_j^{m_{i,j}}$
is left invertible as a consequence of the uniqueness of the 
vector $\left({w_i}/{d}\right)_{i=0}^n=E^{-1}\pmb 1.$
Since we are working over $\CC$, we adopt the notation 
$$[a_0,\dots,a_n]=\diag(\exp(2\pi\cxi a_i))_{i=0}^n$$
for $a_i\in \QQ\cap [0,1[$.
The age of the diagonal matrix above is 
$$\mathrm{age}[a_0,\dots,a_n]=\sum_{i=0}^na_i.$$
The distinguished diagonal symmetry 
$$j_P=\left[\frac{w_0}{d},\dots,\frac{w_n}{d}\right],$$
usually denoted by $j$, spans the intersection $\Aut_P\cap \GG_m$, where 
$\GG_m$ is the group of automorphisms of 
the form $\diag(\la^{w_0},\dots,\la^{w_n})$.
The automorphism $j_P$ is the monodromy operator of the fibration 
defined by 
$W$ restricted to the complement in $\CC^{n+1}$ of the 
zero locus $Z(P)$; we will denote by $M_P$ the generic Milnor fibre
\begin{equation}\label{eq:Milnor}
\xymatrix@R=.5pc{
M_P\ar[rr]\ar[dd] &&\CC^{n+1}\setminus Z(P)\ar[dd]^P\\
&\square&\\
t\neq0 \ar[rr]&&\CC^\times.
}
\end{equation}

For any subgroup $H$ of $\Aut_W$ containing $j_P$ we consider the 
Deligne--Mumford stacks $$\Sigma_{P,H}=[\Sigma_P/H_0],\qquad M_{P,H}=[M_P/H];$$
where $H_0=H/(H\cap \GG_m)=H/\langle j\rangle$ and 
acts faithfully on $\Sigma_P$. The orbifold $\Sigma_{P,H}$ 
is a smooth codimension-$1$ substack of $[\PP(\pmb w)/H_0]$
$$\Sigma_{P,H}\subset [\PP(\pmb w)/H_0],$$
and has trivial canonical bundle as soon as $P$ is Calabi--Yau and 
$H$ lies in 
$$\SL_P:=\Aut_P\cap \SL(n+1;\CC).$$

\subsection{Polynomials with automorphism}\label{sec:polyaut}
More specifically, we focus on 
polynomials of Calabi--Yau type of the form 
$$W(x_0,x_1,\dots,x_n)=(x_0)^k+f(x_1,\dots,x_n)$$
We have 
$\Aut_W=\pmmu_k\times \Aut_f$, 
where, using again the choice $\exp(2\pi i/k)$, the first factor is regarded here 
as $\ZZ/k$, canonically 
generated by 
the order-$k$ automorphism 
\begin{equation}\label{eq:symmetry}
s=\left[\frac1k,0,\dots,0\right].\end{equation}
We have a $\ZZ/k$-action on the stack $\Sigma_{W,H}$ 
$$\mathsf{m}\colon \ZZ/k\times \Sigma_{W,H}\to \Sigma_{W,H}.$$
We have  $j_W=s\cdot j_f$; where $j_f$ is regarded as an element of 
$\Aut_f$. 

Instead of $H\subseteq \Aut W\cap \SL(n+1;\CC)$ 
containing $j_W$, we can equivalently work with subgroups  
$K \subset \Aut_f$ satisfying 
$$(j_f)^k\in K\subseteq \SL_f,$$
(we recover $H$ by considering the subgroup of $\Aut W$ spanned by  $j_W$ and $K$).
More generally we consider the subgroup of $\Aut_W$  
$$K[j_W,s]=\sum_{a,b=0}^{k-1} (j_W)^{a}(s)^{b}K,$$ with its natural  
 $(\frac1k\ZZ/\ZZ)$-gradings 
 $$d_j=\frac{a}{k},\qquad d_s=\frac{b}{k}.$$
By \eqref{eq:CY}, $\exp(2\pi\cxi d_s(g))$ is the determinant of an 
element $g\in K[j_W,s]$. 

\section{Inertia} \label{inertia} 
We consider a finite group $G$ acting on a Deligne--Mumford orbifold $\mf X$ 
$$\mathsf{m}\colon G\times \mf X \to \mf X.$$
We consider the $G$-inertia stack $I_G(\mf X)$
fitting in the following fibre diagram
\[
\xymatrix@R=.5pc{
I_G(\mf X)\ar[rr]\ar[dd] &&\mf X \ar[dd]_{\Delta}\\
&\square&\\
G\times \mf X \ar[rr]_{(\mathsf{m},\mathsf{pr_2})}&& \mf X\times \mf X 
}
\]
When $G$ is a trivial group, $I\mf X$ is the ordinary inertia of $\mf X$. 
There is a locally constant function $$\mf a\colon I\mf X\to \QQ$$
which assigns to each geometric point $(x\in \mf X, g\in \Aut(x))\in I\mf X$
the rational number $\age(g)$ (see  \cite{AGV}
and \cite[\S4.2]{CKV}).

In this way we have 
$$I_G(\mf X)=(G\times \mf X)\times_{(\mathsf{m},\mathsf{pr_2}),\,
\mathfrak X\times \mathfrak X,\,\Delta}\mf X=\bigsqcup_{g\in G} I_g(\mf X)$$
where $I_g(\mf X)$ is the $g$-inertia orbifold 
\begin{equation}I_g(\mf X)\label{eq:g_inertiastack}:=\mathfrak X{\underset{g, \ \mathfrak X\times \mathfrak X, \ \id}{\times}}\mathfrak X.\end{equation} 

The $G$-inertia stack of $\mf X$
fits in the fibre diagram
\[
\xymatrix@R=.5pc{
I_G(\mf X)\ar[rr] \ar[dd]^{\mathsf p} &&\mf X \ar[dd]\\
&\square&\\
 I[\mf X/ G]\ar[rr] && [\mf X/G]
}
\]
and we may  regard $\mathsf p$ as a $G$-torsor by pullback of $\mf X\to [\mf X/G]$. 

The $G$-action on $\bigsqcup_{g\in G} I_g\mf X$ 
is given by conjugation on the indices and by $F_h \colon I_g \mf X \to  
I_{hgh^{-1}} \mf X$ 
on the components, where $F_h$ acts by the effect 
of $h$ on the first factor of \eqref{eq:g_inertiastack}
and by the identity on the second. 

\subsection{A $g$-orbifolded cohomology}
The cohomology of the $g$-inertia stack coincides with the cohomology of 
the $g$-fixed locus when $\mf X$ is representable. 
In the spirit of orbifold cohomology we define 
 $g$-orbifold cohomology groups 
which are invariant under $K$-equivalence. 

The $g$-orbifold cohomology is the cohomology of \eqref{eq:g_inertiastack}
shifted by the locally constant function ``age'' given by 
\begin{equation}\label{eq:age}\mf a_g\colon I_g(\mf X)\to I[\mf X/G]\xrightarrow{\mf a} \QQ.\end{equation}
We assume that $\mf X$ is smooth, so that $I_g(\mf X)$ is smooth and 
all coarse spaces are quasi-smooth; in particular cohomology groups admit a Hodge
decomposition. 
Starting from a Hodge
decomposition of weight $n$, for any $r\in \QQ$, we can produce a new 
decomposition of weight $n-2r$ via 
$\label{eq:Tate}H(r)^{p,q}=H^{p+r,p+r}.$ We will denote by  
$(r)$ the isomorphism induced by the identity at the level 
of the vector spaces; it identifies the Hodge decomposition 
of weight $n$ with the Hodge decomposition of weight $n-2r$.

 We can now provide the definition of $g$-orbifolded cohomology. 

\begin{defn}[$g$-orbifold cohomology]\label{defn:g_orbifoldcoh}
For any $g\in G$ the $g$-orbifold cohomology is defined as
$$H_g^*(\mf X;\CC)=H^*(I_g(\mf X);\CC)(-\mf a_g).$$
\end{defn}

We point out the slight abuse of notation: $\age$ is not constant in 
general, but, since it is locally constant, the shift operates independently on each cohomology group arising from each connected component. A  precise notation should read 
\[ H^{p,q}(\ \ ;\CC)(\mathfrak a)=\bigoplus_{r\in \QQ_{\ge 0}}H^{p,q}(\mathfrak a^{-1}(r);\CC)(-r). \]

For $g=\id=1_G$, the above definition coincides with Chen--Ruan orbifold cohomology
$$H_{\id}^*(\mf X;\CC)=H_{\mathrm{CR}}^*(\mf X;\CC).$$
In this paper, we often consider the relative version of orbifold Chen--Ruan cohomology; indeed 
when $\mf Z$ is a substack of $\mf X$ then $I(\mf Z)$ is a substack of $I(\mf X)$ 
and we set 
$$H_{\id}^*(\mf X,\mf Z;\CC)=H^*(I(\mf X),I(\mf Z);\CC)(-\mf a_{\id}),$$
where $\mf a_{\id}$ is the age function on $I(\mf X)$.

Yasuda \cite{Yasuda} proves the invariance of 
the Hodge decomposition of Chen--Ruan cohomology 
of smooth Deligne--Mumford stacks $\mathfrak X$ and $\mathfrak Y$ 
whenever
there exists a smooth and proper Deligne--Mumuford stack $\mf Z$ with birational 
morphisms $\mf Z\to \mf X$ and $\mf Z\to \mf Y$ 
with $\omega_{\mf Z/\mf X}\cong \omega_{\mf Z/\wt X}$. 
In particular, for Gorenstein orbifolds 
 Chen--Ruan cohomology coincides with 
the cohomology of any crepant resolution of the 
coarse space. 
Furthermore, we have the following proposition.
\begin{pro}\label{pro:crc}
Let $G$ by a finite group acting on a Gorenstein orbifold $\mf X$. Let us assume that
the coarse space $X$ of $\mf X$ admits a crepant resolution $\wt X$ where we 
can lift the $G$-action 
induced by $\mf X$ on $X$. Then, for any $g\in G$ we have 
a bidegree-preserving isomorphism 
$$H_g^*(\mf X;\CC)\cong H_g^*(\wt X;\CC).$$
In particular, the isomorphism identifies $H_g^*(\mf X;\CC)$ with 
$H^*(\wt X_g;\CC)(-\wt{\mf a}_g)$,
where $\wt{\mf a}_g$ is the composite $\wt X_g\to [\wt X/G]$ and of the age function 
$[\wt X/G]\to \QQ$. 
\end{pro} 
\begin{proof}
The stack $\mf X$ and its resolution $\wt X$ are $K$-equivalent.
In order to see this, we  consider the 
 $\mf Z=\mf X\times _X \wt X$ and the associated reduced stack. 
Then, there exists a proper birational morphism $\mf Z' \to \mf Z$ such that $\mf Z'$ is smooth. 
This is explained in Sect.~4.5, \S2, of Yasuda's paper \cite{Yasuda} 
(this is essentially due to Villamayor papers \cite{Vil1} and \cite{Vil2} showing  the existence of 
resolutions compatible with smooth, in particular \'etale, morphisms). Actually, 
in his recent generalization \cite{YasudaWild}, Yasuda 
proves that it suffices to consider the reduction and the 
normalization of $\mf Z'$, without
 any resolution. This happens because his new statements allows us to
extend the definition 
of orbifold cohomology
to singular or wild (in positive characteristic) 
Deligne--Mumford stacks.

Now we consider the abelian group $H=\langle g\rangle$. Then 
$\mf A'=[\mf X/H]$ and 
$\mf A''=[\wt X/H]$ are $K$-equivalent by the same argument. Indeed the action of $H$ 
descends compatibly to 
the coarse space $X$ and we can consider the stack $\mf A=[X/H]$ and 
the morphisms $\mf A'\to\mf A$ and $\mf A''\to \mf A$. Then, the reduced stack associated
to the fibred product 
$\mf A'\times_{\mf A} \mf A''$ can be resolved and yields a smooth Deligne--Mumford 
stack $\mf Z$ mapping to $\mf A'$ and $\mf A''$. As above, the fact that 
the canonical bundles of 
$\mf X$ and $\wt X$ are the pullback of $\omega_X$ is enough to show that $\mf Z\to 
\mf A'=[\mf X/H]$ and $\mf Z\to \mf A'=[\mf X/H]$ is a $K$-equivalence.

The desired claim follows because the 
cohomology of $I_g\mf X$ and that of $I_g\wt X$ appear as 
summands of the Chen--Ruan cohomology groups of $[\mf X/H]$ and of $[\wt X/H]$. 
Indeed they arise as the cohomology groups of the 
sectors attached to $g$ whose cohomology are 
the $g$-invariant classes of $I_g(\mf X)$ and $I_g(\wt X)$. 
Since $g$ operates trivially on these sectors, we can regard these contributions as $H^*(I_g(\mf X); \CC)$  and $H^*(I_g(\wt X);\CC)$. We should further mention 
that we obtain an identification at the level of the age-shifted $g$-orbifolded cohomology $H_g^*(\ \ ;\CC)$ due to the fact that 
the age  is a rational function factoring through the 
age function of $[\mf X/H]$ and of $[\wt X/H]$.
\end{proof}
\begin{rem}
The proposition above only claims the existence of an isomorphism. In special cases 
in dimension $2$ we have proven the existence of an explicit isomorphism, see \cite{CKV}.  
\end{rem}
In special cases where 
$\wt{\mf a}_g$ is constant, the above theorem 
allows us to relate the $g$-orbifold cohomology to the cohomology
of the $g$-fixed locus of the resolution via a constant shift by $\wt{\mf a}_g$.
The following example generalises the case of 
anti-symplectic involutions of orbifold K3 surfaces considered in \cite{CKV} 
(this case occurs below for $k=2$).
\begin{exa}\label{eg:dim2yasuda}
Consider a proper, 
smooth, Gorenstein, Deligne--Mumford orbifold $\mf X$ of dimension 2 satisfying the
 Calabi--Yau condition $\omega\cong \Ocal$. We refer to this as a K3 orbifold because 
 there exists a minimal resolution $\wt X$ which is a K3 surface. 
 Consider the volume form $\Omega$ of $\wt X$, which descends on $\mf X$. 
 We assume that $g$ is an order-$k$ automorphism of $\mf X$ whose induced action on  $\Omega$ 
 is multiplication by $e^{2 \pi i (k-1)/k}$. Then, 
 $g$ naturally lifts to the minimal resolution $\wt X$; furthermore, locally at each fixed point of $\wt X$,
 the action of $g$ can be written as $\frac1k[a,b]$ with $a+b=k-1$ (this happens because 
 the case $a+b=2k-1$ is impossible). In this way the age shift $\mathfrak a_g$ at the fixed loci always equals $1-1/k$ 
$$H_g^*(\mf X;\CC)=H^*(\wt X_g ;\CC)(\textstyle{\frac1k-1}).$$
\end{exa}

\section{Landau--Ginzburg state space} \label{LG}
The expression ``Landau--Ginzburg'' comes from 
physics and 
is often used for $\CC$-valued functions defined 
on vector spaces possibly equipped with the action of a group. More generally
the definition is 
extended to vector bundles on a stack. 
In this paper we only use it for the above setup $P\colon [\CC^{n+1}/H]\to \CC,$
where $P$ is a non-degenerate polynomial and $j\in H\subseteq \Aut_P$. 
Indeed this may be regarded as a $\CC$-valued function defined on 
a rank-$(n+1)$ vector bundle on the stack $BH=[\Spec \CC/H]$. 
We show how this geometric setup is 
naturally connected to $\Sigma_{W,H}$ via $K$-equivalence. 

\subsection{$K$-equivalence}\label{subsect:Kequivalence}
Consider the rank-$(n+1)$ vector bundle 
$$\VV=\Ocal_{\PP(d)}(-w_0)\oplus \dots\oplus \Ocal_{\PP(d)}(-w_n)=[\CC^{n+1}/\langle j\rangle];$$
its coarse space $X=\CC^{n+1}/\langle j\rangle$, and the smooth Deligne--Mumford stack  
$$\mathbb L=\Ocal_{\PP(\pmb w)}(-d)\to X,$$
total space $\LL$ of the line bundle of degree $-d$ on $\PP(\pmb w)$. The 
stacks $\VV$ and $\LL$ are the  two GIT quotients of $\CC\times\CC^{n+1}$ modulo $\GG_m$ operating with weights $(-d,w_0,\dots,w_{n+1})$.
Notice that $\VV$ without the origin coincides with the line bundle 
$\LL$ without the zero section: $\VV^\times =\LL^\times$. 

We assume that $P$ is of Calabi--Yau type in the sense of \eqref{eq:CY}. 
Then, the canonical bundle of $\VV$ descends to $X$ and its pullback 
to $\LL$ coincides with $\omega_\LL$. 
Following the same argument as above, by Yasuda \cite{Yasuda},
we have 
\begin{equation}\label{eq:Yasuda}
\Phi\colon H^{p,q}_{g}(\VV;\CC)\xrightarrow{\cong} H^{p,q}_{g}(\LL;\CC)\end{equation}
for any $p,q\in\QQ$ and for any $g\in \Aut_W$. 

The isomorphism $H^d_{g}(\VV;\CC)\to H^d_{g}(\LL;\CC)$ 
is not canonical; notice, however, that 
we can at least impose a compatibility with respect to the 
restrictions 
$H^d_{g}(\VV;\CC)\to H^d_{g}(\FF;\CC)$ and 
$H^d_{g}(\LL;\CC) \to H^d_{g}(\FF;\CC)$  
for $\FF=[M_P/\langle j \rangle]$ included in $\VV^\times = \LL^\times\subseteq \VV,
\LL$.
This happens because the fundamental classes of the 
inertia stacks $I_g({\VV})$, $I_g(\LL)$ and $I_g(\FF)$ 
attached to the same automorphism 
$\beta=g\cdot(\la^{-d},\la^{w_1},\dots,\la^{w_n})$ with
 $\la\in \cup_{j=1}^m \pmmu_{w_j}$ can be identified
 since their bidegree equal $(\age(\beta),\age(\beta))$ by construction. In this way,
 we can require that \eqref{eq:Yasuda}
 respects the canonical identification between the fundamental 
 classes of $I_g({\VV})$ and $I_g(\LL)$ and this is enough to 
 insure that the following diagram commutes
\[
\xymatrix@C=1.1pc{
\dots \ar[r] 
&H^{d}_{g}(\VV,\FF;\CC)\ar[r]\ar[d]& H^d_{g}(\VV;\CC)\ar[d]_{
\Phi} \ar[r]& 
H^d_{g}(\FF;\CC)\ar[d]_=\ar[r]& H^{d+1}_{g}(\VV,\FF;\CC)\ar[r]\ar[d]&\dots   \\
\dots\ar[r]&H^{d}_{g}(\LL,\FF;\CC)\ar[r]& H^d_{g}(\LL;\CC) \ar[r]& 
H^d_{g}(\FF;\CC)\ar[r]& H^{d+1}_{g}(\LL,\FF;\CC)\ar[r]&\dots  
}
\]
 and yields a bidegree-preserving isomorphism 
 $$H^d_{g}(\VV,\FF;\CC)\cong H^d_{g}(\LL,\FF;\CC).$$
 Note that $\FF$ can be regarded as 
 the generic fibre of $P\colon \VV\to \CC$ as well as the 
 generic fibre of $P\colon \LL\to \CC$. 
 If we consider any group $H\subseteq \Aut_P$ containing $j$
 we can apply the above claim to $\VV_H=[\VV/H_0]$, 
 $\LL_H=[\LL/H_0]$ and $\FF_{P,H}
 =[\FF/H_0]$. We get 
 \begin{equation}\label{eq:LGCY_Kequiv}
H^{p,q}_{g}(\VV_{H},\FF_{P,H};\CC)\cong H^{p,q}_{g}(\LL_H,\FF_{P,H};\CC)
 \end{equation}
 for any $p,q$ and for any $g\in \Aut_W$. 
 
   The left hand side is naturally identified via the Thom isomorphism to the Chen--Ruan cohomology of 
  $\Sigma_{P,H}$ up to a (-1)-shift whereas the left hand side is naturally identified to 
  an orbifold version of the Jacobi ring known as the FJRW 
  or Landau--Ginzburg state space.
 We detail these two aspects in the next two sections.

\subsection{Thom isomorphism}
Consider $P\colon \LL\to \CC$
and its generic fibre $\FF=P^{-1}(t)$ for $t\neq 0$.
We have an isomorphism of Hodge structures
\begin{equation}\label{eq:Thom}
H^*(\LL, \FF;\CC)\cong H^*(\Sigma_{P};\CC)(-1).\end{equation}
This happens because the left hand side can be regarded after retraction as 
\[H^*(\PP(\pmb w),\PP(\pmb w)\setminus \Sigma_P;\CC)\] which is isomorphic 
to the $(-1)$-shifted cohomology of $\Sigma_P$ by the Thom isomorphism.

Equation \eqref{eq:Thom} suggests that the orbifold cohomology $H^{p,q}_{\id}(\LL_H,\FF_{P,H};\CC)$ is 
related to the orbifold cohomology of $\Sigma_P$. 
However, the argument above does not yield an isomorphism respecting the 
orbifold cohomology bidegree. 
This happens because 
\[H^{*}(\LL_H,\FF_{P,H};\CC)\cong H^*(\PP(\pmb w),\PP(\pmb w)\setminus \Sigma_P;\CC)\cong H^*(\Sigma_{P};\CC)(-1)\]
may fail in orbifold cohomology. 
However the first author and Nagel proved that  
Equation \eqref{eq:Thom} holds in orbifold cohomology without changes even 
when the Thom isomorphism
$H_{\id}^*(\PP(\pmb w),\PP(\pmb w)\setminus \Sigma_P;\CC)\cong H_{\id}^*(\Sigma_{P};\CC)(-1)$
does not. We regard 
$$H_{\id}^*(\LL, \FF;\CC)\cong H_{\id}^*(\Sigma_{P};\CC)(-1)$$
as the correct formulation of Thom isomorphism in orbifold cohomology. We refer to 
the theorem below.

For the benefit of the reader we illustrate the study 
of $H_{\id}^*(\PP(\pmb w),\PP(\pmb w)\setminus \Sigma_P;\CC)$ and 
of $H_{\id}^*(\Sigma_{P};\CC)(-1)$
sector-by-sector. We 
distinguish two cases. 
For ordinary sectors (such as the untwisted sector) 
$\Sigma_{P,g}$ is a codimension-$1$ 
hypersurface in $\PP(\pmb w_g)$. Then, there is an identification 
$H^*(\PP(\pmb w_g),\PP(\pmb w_g)\setminus \Sigma_{P,g};\CC)\cong H^*(\Sigma_{P,g};\CC)(-1)$ and the age shift of $\PP(\pmb w_g)$ coincides with that of $\Sigma_{P,g}$ since 
$g$ acts trivially on the normal bundle $N_{\Sigma_{P,g}/\PP(\pmb w_g))}$. 
On the other hand, it may  happen 
that $\PP(\pmb w_g)$ and $\Sigma_{P,g}$ coincide as we illustrate in Example \ref{exa:nonGorenstein}. 
In these cases we have 
$H^*(\PP(\pmb w_g),\PP(\pmb w_g)\setminus \Sigma_{P,g};\CC)=
H^*(\PP(\pmb w_g);\CC)\cong H^*(\Sigma_{P,g};\CC)$
without Tate shift $(-1)$.
Furthermore the difference between the age shift of the sector 
$\PP(\pmb w_g)$ and that of $\Sigma_{P,g}$ is strictly positive: it equals the age 
$q\in ]0,1[$ of the character 
operating via $g$ on $N_{\Sigma_{P,g}/\PP(\pmb w_g))}$. 
In these cases we have  
$H^*(\PP(\pmb w_g),\PP(\pmb w_g)\setminus \Sigma_{P,g};\CC)\cong H^*(\Sigma_{P,g};\CC)(-q)$. It is now possible to observe that we have 
$$H^*_{\id}(\LL_g,\LL_g\setminus \Sigma_{P,g};\CC)\cong H_{\id}^*(\Sigma_{P,g};\CC)(-q-(1-q))=
H_{\id}^*(\Sigma_{P,g};\CC)(-1)$$
as desired. The result is proven in 
\cite{ChiodoNagel} for all complete intersections. We get 
\begin{thm}[Thom isomorphism, \cite{ChiodoNagel}] 
For any $H\subseteq \Aut_P$ containing $j_P$ and $g\in \Aut_P$ and 
for any $p,q\in \QQ$, we have 
\begin{equation}\label{eq:OrbiThom}
H^{p,q}_{g}(\LL_H, \FF_{P,H};\CC)\cong H^{p,q}_{g}(\Sigma_{P,H};\CC)(-1).
\end{equation}
\end{thm}

The following example is added here in the sake of clarity, 
but plays no essential role in the rest of the text; 
it illustrates in a simple way the issue arising for 
a Calabi--Yau embedded in a nonGorenstein $\PP(\pmb w)$.
\begin{exa}\label{exa:nonGorenstein}
Consider the hypersurface $\Sigma$ defined by $x^3+xy=0$ in $\PP(1,2)$. It consists of 
two points, the orbit $p=(-x^2,x)$ and the point $(x=0)=\PP(2)$. 
We have $I(\Sigma)=\Sigma_0\sqcup \PP(2)_1$ and $I(\PP(1,2))=\PP(1,2)_0\sqcup \PP(2)_1$, 
where the label $j=0,1$ denotes the root of unity $\exp(2\pi \cxi  j/2)$ 
acting as $(\la,\la^2)$ with nonempty fixed locus. The orbifold cohomology of $\Sigma$ is 
$\CC^3$ concentrated in degree $(0,0)$ (the age-shift does not intervene
for a $0$-dimensional stack). Since the 
Thom isomorphism holds after a $(-1)$-shift, we compare 
$\CC^3(-1)$ to $$H^*_{\id}(\PP(1,2),\PP(1,2)\setminus \Sigma)=
\CC^2(-1)\oplus H^*(\PP(2))(-\textstyle\frac12)\cong \CC^2(-1)\oplus \CC(-\textstyle\frac12).$$ 
On the other hand, we set $\FF$ via $P=x^3+xy$ as above; then $\CC^3(-1)$ matches  
$$H^*_{\id}(\LL,\FF)=
\CC^2(-1)\oplus H^*({\PP(2)})(-\textstyle 1)\cong \CC^3(-1).$$ 
We refer to \cite[Prop.~3.4]{ChiodoNagel}.
\end{exa}

\begin{rem}\label{rem:ambient} 
The ambient cohomology of $\Sigma_{P,H}$ is Poincar\'e dual to the image of the 
the homology of 
 $I_g(\Sigma_{P,H})$ within the homology of $I_g(\PP(\pmb w))$. By the identification above 
we may regard it also as 
the image of 
the morphism 
\begin{equation}\label{eq:restrictionforL}H^k_{g}(\LL_H, \FF_{P,H};\CC)\to H^k_{g}(\LL_H;\CC).\end{equation}
We can also consider the primitive cohomology 
of $I_g(\Sigma_{P,H})$ whose direct image vanish  
in $I_g(\PP(\pmb w))$. Than the kernel of the 
morphism \eqref{eq:restrictionforL} above matches 
the primitive cohomology of $\Sigma_{P,H}$ in 
$H^*(\Sigma_{P,H};\CC)(1)$
\begin{align*}
H_{g,\mathrm{amb}}^*(\Sigma_{W,H};\CC)(-1)&=\im\big(H^k_{g}(\LL_H, \FF_{P,H};\CC)\to H^k_{g}(\LL_H;\CC)\big),\\
H_{g,\mathrm{prim}}^*(\Sigma_{W,H};\CC)(-1)&=\ker\big(H^k_{g}(\LL_H, \FF_{P,H};\CC)\to H^k_{g}(\LL_H;\CC)\big).
\end{align*}
We now turn to the LG side, where 
the image of the 
analogue morphism allows us to describe the so called ``narrow'' and ``broad'' sectors. 
\end{rem}

\subsection{Jacobi ring}\label{sect:Jacobiring} The 
Jacobi ring $$\Jac P = dx_0\wedge\dots\wedge dx_n\CC[x_0,\dots, x_n]/ 
(\partial_0P, \dots, \partial_n P),$$
regarded as 
a $\CC$-vector space, has 
dimension $\prod_j \frac{d-w_j}{w_j}$ (due to the non-degeneracy of the polynomial $P$) and 
is isomorphic to $H^*(\CC^n,M_P;\CC)$.
The natural monodromy action of $\pmmu_d=\langle j\rangle$ 
from \eqref{eq:Milnor}, and more generally 
the action of any $\diag(\al_0,\dots, \al_n)\in \Aut_P$,
$$\diag(\al_0,\dots, \al_n)\cdot 
\left(\prod_{j=0}^n z_j^{b_j-1} \bigwedge_{j=0}^n dz_j\right)
= \prod_{j=0}^n \al _j^{b_j}  \left(\prod_{j=0}^n z_j^{b_j-1} 
\bigwedge_{j=0}^ndz_j\right),$$
allows us to write 
$$[\Jac(f)^{j_f}]_{p,q}=H^{p,q}(\VV;P^{-1}(t)),$$
where the subscript ${p,q}$ denotes the elements 
$$\left(\prod_{j=0}^n z_j^{b_j-1} 
\bigwedge_{j=0}^n dz_j\right)\quad \text{with \quad
$(p,q)=\left(n-\sum_j b_j\frac{w_j}d ,\sum_j b_j\frac{w_j}d \right)$}.$$
The above claim is due to Steenbrink \cite{Steenbrink}
in the present weighted homogenous setup, see also 
\cite[Appendix A]{ChiodoIritaniRuan}. 

\begin{rem}\label{rem:grading}
The action of $\Aut_P$ on $\Jac(P)$ is well defined because 
any automorphism $\diag(\al_0,\dots,\al_n)$ operates on 
each monomial in $\partial_i P$ by multiplication by $\al_i^{-1}$. 
This happens because $\diag(\al_0,\dots,\al_n)$ fixes 
$x_i\partial_iP$ since it fixes $P$. 

Furthermore, the grading $(p,q)$ is well defined 
simply because $\deg(x_i)=\frac{w_i}d$ yields 
a $\QQ$-grading on $\CC[x_0,\dots, x_n]$, 
which descends to a $\QQ$-grading of $\Jac(P)$ 
because the Jacobi ideal $(\partial_0P, \dots, \partial_n P)$ 
is homogeneous (each monomial in $\partial_i P$ 
has degree $d-\frac{w_i}d$).  
\end{rem}

This calls for the following definition.
\begin{defn}\label{defn:hcal}
For a quasi-homogenous polynomial $P$ of degree $d$ and weight $w_0,\dots,w_n$ and for any $H\subseteq \Aut_P$ containing $j_P$,
the $g$-orbifolded  
Landau--Ginzburg state space is 
$$\mathcal H_{P,H,g}=
\bigoplus_{h\in g H} \left(\Jac P_{h}\right)^H(-\age(h)),$$
where, for any diagonal symmetry $h\in H$, we consider 
the Jacobi ring 
$\Jac P_{h},$
where $P_{h}$ is the  restriction of 
$P$ to the ring of polynomials in the 
$h$-fixed variables. \end{defn}

\begin{rem}Notice that, as a consequence of the non-degeneracy of 
$P$, the restriction $P_g$ 
is still a non-degenerate polynomial.
\end{rem}
\begin{rem}
When $g$ is the identity 
we recover the FJRW state space $\mathcal H_{P,H}$.\end{rem}

%The state space is naturally bi-graded: we can write a state 
%within the summand $\Jac P_g$ where $g=[p_0, p_1,\dots, p_n]$, 
%as $\prod_{i\in I} x_i^{b_i-1}\bigwedge_{i\in I} dx_i$ for 
%$i\in I$ if and only if $p_i=0$; then we assign the 
% bidegree 
%$$(p,q)=\left(\#I-\sum_{i\in I} b_i\frac{w_i}d + \sum_{i\not \in I} p_i,\quad \sum_{i\in I} b_i\frac{w_i}d + \sum_{i\not\in I} p_i\right).$$
\begin{rem}
We immediately have  
\begin{equation}\label{eq:statespace_alg_geom}
\mathcal H _{P,H,g}^{p,q} =H_{g}^{p,q}(\VV_H;\FF_{P,H}).\end{equation}\end{rem}
\begin{rem}\label{rem:narrow}
The elements $h$ for which no variables are fixed 
yield a summand $\Jac(P_{h})=\CC$; these are special elements in 
the FJRW state space; they span the subspace of the 
so called narrow classes. In FJRW theory, 
the remaining summands are referred to as 
broad classes. 
In complete analogy with  
ambient cohomology, we can identify the narrow and broad classes to 
the image and the kernel of 
the morphism 
$$H^k_{g}(\VV_H, \FF_{P,H};\CC)\to H^k_{g}(\VV_H;\CC).$$
We have
\begin{align*}
\Hcal_{P,H,g}^{*,\mathrm{nar}}=\im\big(H^k_{g}(\VV_H, \FF_{P,H};\CC\big)\to H^k_{g}(\VV_H;\CC)),\\
\Hcal_{P,H,g}^{*,\mathrm{broad}}=\ker\big(H^k_{g}(\VV_H, \FF_{P,H};\CC\big)\to H^k_{g}(\VV_H;\CC)).
\end{align*}
\end{rem}
%FIXME

\subsection{Landau--Ginzburg/Calabi--Yau correspondence}
The above equations \eqref{eq:LGCY_Kequiv}, \eqref{eq:OrbiThom}, and \eqref{eq:statespace_alg_geom} add up to a simple 
proof of the so called Landau--Ginzburg/Calabi--Yau correspondence based on Yasuda's theorem and $K$-equivalence (insured by the Calabi--Yau condition). 
\begin{thm}[\cite{ChiodoRuan, ChiodoNagel}]\label{thm:lgcy}
For any non-degenerate 
quasi-homogeneous polynomial $P$ of Calabi--Yau type,
for any group $H\subseteq \Aut_P$ containing $j_P$, and for any $g\in \Aut_P$,
we have 
$$\Phi\colon H_g^{p,q}(\Sigma_{P,H};\CC)(-1)
\xrightarrow{ \ \sim\  } \Hcal_{P,H,g}^{p,q}.$$
\end{thm}
Since the above isomorphism follows from 
$K$-equivalence, it is not explicitly given. 
In  \cite{ChiodoRuan} we provide an explicit automorphism. 
In \cite{ChiodoNagel} we generalize it to complete 
intersections.
Notice that, by a slight abuse of notation, we adopted the same notation for 
the above isomorphism as for $\Phi\colon H_g^*(\VV;\CC)\to H_g^*(\LL;\CC)$ from 
\eqref{eq:Yasuda}.

%
%
%(in particular \cite{ChiodoNagel} provides a 
%generalization of the isomorphism 
%\eqref{eq:statespace_alg_geom}
%relating orbifold cohomology and Jacobi rings which does not 
%follow from $K$-equivalence).

\section{Unprojected mirror symmetry} \label{sect:unprojected}
\subsection{Mirror duality}
The mirror construction due to Berglund and H\"ubsch \cite{BH} is elementary. 
It applies to non-degenerate polynomials $P$ of invertible type, \emph{i.e.}
having as many monomials as variables. Up to rescaling the variables 
these polynomials are entirely encoded by the 
exponent matrix $E=(m_{i,j})$, and are paired to a second polynomial of the same type
\[P(x_0,\dots, x_n)=\sum_{i=0}^n \prod_{j=0}^n x_j^{m_{i,j}},\qquad 
P^\vee(x_0,\dots, x_n)=\sum_{i=0}^n\prod_{j=0}^n  x_j^{m_{j,i}},\]
by transposing the matrix of exponents $E$. 

\begin{rem}
The square matrix $E$ is invertible 
because the vector $(w_i/d)_{i=0}^n$ is uniquely determined as a 
consequence of the 
non-degeneracy condition.
The inverse matrix $E^{-1}=(m^{i,j})$ allows a simple description of $\Aut_P$: 
the columns express the symmetries $$\rho_j=[m^{0,j},m^{2,j},\dots,m^{n,j}]$$ spanning 
$\Aut_P$. It is also easy to see that 
the columns of $E$ express all the relations $\sum_i m_{i,j} \rho_i$
among these generators. Naturally, the rows of $E^{-1}$ provide an expression 
for the symmetries $\ol \rho_i$ generating $\Aut_{P^\vee}$ under the relations 
provided by the rows $\sum_j m_{i,j} \ol \rho_j$ of $E$. 
In particular we have a canonical isomorphism 
$$\Aut_{P^\vee} = (\Aut_P)^*,$$
where $(G)^*$ denotes the Cartier dual $\Hom(G,\GG_m)$. 
The identification matches the symmetry $[q_0,\dots, q_n]$ to 
the homomorphism mapping $\rho_i$ to  
$\exp(2\pi \cxi q_i) \in \QQ/\ZZ$. 
Based on this identification, for any subset $S\subseteq \Aut_P$, 
we set $S^\vee \subseteq \Aut_{P^\vee}$ as follows 
$$S^\vee= \{\fie \in (\Aut_P)^*\ \mid \ \fie\!\mid_S=0\}.$$
 This is a duality exchanging subgroups of $\Aut_P$ and 
subgroups of $\Aut_{P^\vee}$; for any group $H\subseteq \Aut_P$,
we can write 
\begin{equation}\label{eq:groupduality}
H^\vee=\ker \left(\Aut_{P}\twoheadrightarrow \Hom(H;\GG_m)\right).\end{equation}
It exchanges 
$J_P$ with $\SL_{P^\vee}$.
It reverses the inclusions. 
\end{rem}

We define the unprojected state space 
$$U_P=\bigoplus_{h\in \Aut_P} \Jac(P_h)(-\age(h))$$
by summing over all diagonal symmetries and without taking any invariant. 
For each summand $\Jac(P_h)$ 
there exists an $\Aut_{P^\vee}$-grading 
defined by 
\begin{align}
\ell_h\colon \Jac(P_h)&\to \Aut_{P^\vee}\nonumber\\
 \prod_j x_j^{b_j-1}\bigwedge_{j} dx_j&\mapsto 
\prod_j \ol \rho_j^{b_j},\label{eq:ellhmap}
\end{align}
where all the products run over the set $F_h$ of labels 
of variables fixed by $h$.
The map is well defined because each monomial in $\partial_iP_h$ 
maps to $J_{P^\vee}(\ol \rho_i)^{-1}$. This happens because each monomial appearing
in $\partial_iP_h dx_i$ 
maps to the same automorphism as each monomial appearing in $P$.
Furthermore, the automorphism obtained in this way is the identity
by the relation $\sum_j m_{i,j} \ol \rho_j$
discussed above. 
We can finally conclude that $\partial_iP_h \bigwedge_j dx_j$ 
maps to the same automorphism as $\bigwedge_{j\neq i} dx_j$: namely 
$\prod_{j\neq i} \ol \rho_j=
J_{P^\vee}(\ol \rho_i)^{-1}$.
In this way the unprojected state space admits a double decomposition
\begin{equation}\label{eq:doubledecomp}U_P=\bigoplus_{h\in \Aut_P} \Jac(P_h)(-\age(h))=
\bigoplus_{h\in \Aut_P}\ \bigoplus_{k\in \Aut_{P^\vee}} 
U_h^k(P),\end{equation}
where $U_h^k(P)$ is the $k$-graded component of $\Jac(P_h)(-\age(h))$.
We write 
$$U_H^K(P)=\bigoplus_{h\in H}\ \bigoplus_{k\in K} 
U_h^k(P),$$
for any set $H\subseteq \Aut_P$ and $K\subseteq \Aut_{P^\vee}$.
When a subscript $H$ or a supscript $K$ is omitted we assume 
that $H$ or $K$ equal $\Aut_P$ or $\Aut_{P^\vee}$.
When no ambiguity may occur, we omit the polynomial $P$ in the notation.
\begin{pro}\label{pro:invariants}
The vector space 
$U_H^K$ is the $K^\vee$-invariant subspace of $U_H$
$$U_H^K(P)=\left[U_H(P)\right]^{K^\vee}.$$
In particular we have 
$$\Hcal_{P,H,g}= U_{gH}^{H^\vee}(P).$$
\end{pro}
\begin{proof}
This happens because, for any form $f$ in $\Jac(P_h)(-\age(h))\subseteq U_P$
the following equivalence holds. We have $\ell_h(f)\in K$ if and only if 
$f$ is invariant with respect to $K^\vee$. This is just another way to phrase 
the definition of $K^\vee$. 
\end{proof}

\begin{thm}[Krawitz \cite{Krawitz}, Borisov \cite{Borisov}] 
For any $h\in \Aut_P$ and $k\in \Aut_{P^\vee}$, 
we have an explicit isomorphism 
$$U_h^k(P)\cong U_k^h(P^\vee),$$
yielding an explicit isomorphism 
\begin{align*}
M_P\colon U_P &\longrightarrow U_{P^\vee}
\end{align*}
mapping $(p,q)$-classes to $(n+1-p,q)$ classes. 
\end{thm}
We illustrate the isomorphism explicitly in the special case where 
$P=x^k$. It is elementary and it  plays a crucial role in this paper. 
\begin{exa}\label{exa:xk}
Let $P=x^k$. Then $\Aut_P$ equals $\ZZ_k$ (because we fix a primitive $k$th root $\xi$). 
For $1\in \pmmu_k=\ZZ_k$, we have $P_1=P$, so $$\Jac(P_1)=dx\CC[x]/(x^{k-1})= \sum_{h=1}^{k-1} U_{1}^{\xi^h}$$
with $U_{1}^{\xi^h}$ spanned by $x^{h-1}dx$.
Furthermore, for $i\neq 0$, we have $P_{\xi^i}=0$, so $$\Jac(P_{x^i})=\CC=U_{\xi^i}^1\ \ \text{(for $i=1,\dots, k-1$)}.$$ By mapping $x^{i-1}dx\in \Jac(P_1)$ to the 
generator $1_{\xi^i}$ of $\Jac(P_{x^i})$ we have defined a map matching $(1-\frac{i}k,\frac{i}k)$-classes 
to  $(\frac{i}k,\frac{i}k)$-classes.
\end{exa}

If $P$ is a polynomial which 
can be expressed as the sum of two invertible and non-degenerate 
polynomials $P'$ and $P''$ involving disjoint sets of variables we clearly have 
$\Aut_P=\Aut_{P''}\times \Aut_{P''}$. This and the  theorem  above
imply the following crucial properties of the mirror map $M_P$. 
\begin{description}
\item[Thom--Sebastiani.]
If $P$ is a polynomial which 
can be expressed as the sum of two invertible and non-degenerate 
polynomials 
$$P=P'(x'_0,\dots,x'_{n_1})+P''(x''_0,\dots,x''_{n_2})$$ 
involving two disjoint sets of variables,
then we have $$M_P=M_{P'}\otimes M_{P''}.$$
\item[Group actions.]
For any $H, K \subseteq \Aut_P$, the restriction of $M_P$
yields an isomorphism  
\begin{equation}M_P\colon  U_H(P)^K \cong U_{H^\vee}(P^\vee)^{K^\vee}.
\end{equation}
\end{description}
This mirror construction is due to Berglund and H\"ubsch \cite{BH}. The result presented here 
appeared first in this form in Krawitz \cite{Krawitz}; 
we should also refer  to Berglund and Henningson \cite{BerglundHenningson} for 
the group duality, to Kreuzer and Skarke \cite{KreSka} for a systematic study,
and to Borisov \cite{Borisov} for a  reinterpretation 
of the setup and further generalizations in terms of vertex algebr\ae.

There are many consequences of the existence of $M_P$ and of its 
properties with respect to group actions. We list a few of them, starting from 
the first, most transparent, application. 
It appeared in \cite{ChiodoRuan} and it should be regarded as 
a combination of the mirror map $M_P$ of Krawitz and Borisov \cite{Krawitz, Borisov} and of 
the LG/CY isomorphism $\Phi$ of the first named author with Ruan \cite{ChiodoRuan}. In the present setup, it is extremely elementary to prove its main 
statement.
\begin{thm}[Mirror Symmetry for CY models]\label{thm:MSCY}
For any 
invertible, non-degenerate $P$ of Calabi--Yau type and 
for any $H\subseteq\Aut_P$ satisfying $j_P\in H\subseteq \SL_P$, we have
an isomorphism 
$$H_{\id}^{p,q}(\Sigma_{P,H};\CC)\cong H_{\id}^{n-1-p,q}(\Sigma_{P^\vee,H^\vee};\CC) \qquad (p,q\in \QQ).$$
\end{thm}
\begin{proof}
Since $M_P$ satisfies the above property with respect to group actions we have 
$M_P(\Hcal_{P,H,\id})=\Hcal_{P^\vee,H^\vee,\id}$. In order to apply Theorem \ref{thm:lgcy}
we need the Calabi--Yau condition on $P$ and the conditions $H\ni j_P$ and $H^\vee \ni j_{P^\vee}$. The last equation is equivalent to 
$H^\vee\subseteq\SL_P$. The claim follows. 
\end{proof}
The Thom--Sebastiani property 
applies to $P'=x_0^k$ and $P''=f$ adding up to $$W=x_0^k+f(x_1,\dots, x_n).$$ 
The aim of this paper is to study the relation between the above
cohomological mirror symmetry and the the symmetry $s=[\frac1k,0,\dots,0]$.
\begin{pro}\label{pro:movingbecomefixed}
Let $(\phi,h)$ be a monomial element
$$(\phi,h)=\left(\prod_{j=0}^n z_j^{b_j-1} 
\bigwedge_{j=0}^n dz_j\right)$$
 in $\Jac(P^\vee_h)(-\age(h))$. Let $\xi=\exp(2\pi\cxi /k)$. Then, 
$s^*(\phi,h)=\xi^i(\phi,h)$ if and only if $M_W(\phi,h)$ is of the form 
$(\phi',h')$ with $h'=[\frac{i}k,a_1,\dots,a_n]$. In particular 
$M_W$ maps invariant elements to non-invariant elements.
\end{pro}
\begin{proof} This happens because $s$ spans $\Aut_{x_0^k}$, whose dual group is 
trivial.  The claim follows by $M_{W}=M_{x_0^k}\otimes M_f$ (see Example \ref{exa:xk}).
\end{proof}

\subsection{Unprojected states and automorphisms} 
We study the behaviour of $U_W$ with respect to $s$. 
We begin by restricting to a conveniently large state space $\HH_W$ within $U_W$.
 
Let us consider 
$W=x_0^k+f(x_1,\dots,x_n)$ and, as in \S \ref{sec:polyaut}, a
subgroup $K \subset \Aut_f$ satisfying 
$$(j_f)^k\in K\subseteq \SL_f.$$ 
Set 
$$\HH_{K[j_W,s]}^K(W)=\left[U_{K[j_W,s]}(W)\right]^K=\bigoplus_{h\in K[j_W,s]} \Jac(P_h)^K(-\age(h)).$$
If no ambiguity arises, when the polynomial $W$ and the group $K$ are fixed, we write simply $\HH$. 

In the above setup we have three groups: $K$, $K[j_W]$ and $K[j_W,s]$. Only $K[j_W]$ satisfies the 
two conditions of mirror symmetry theorems: namely, it contains $j_W$ and is contained in $\SL_W$.
Its mirror group $K[j_W]^\vee$ has the same properties. 
The following 
proposition (of immediate proof) describes how $K$ and $K[j_W,s]$ behave with respect to 
the the group duality.
\begin{pro}\label{pro:MS}
Consider $K\in \Aut_f$ satisfying 
$(j_f)^k\in K\subseteq \SL_f.$ Then we have $$(j_{f^\vee})^k\in (K[j_W,s])^\vee\subseteq \SL_{f^\vee}\quad\text{and}\quad K^\vee=(K[j_W,s])^\vee[j_{W^\vee},s].$$
Furthermore, we have a mirror isomorphism 
\begin{equation}\label{eq:unprojected_mirror}
M_W\colon \left(\Hbar_{K[j_W,s]}^K(W)\right)^{p,q}=
\left(\Hbar_{K^\vee}^{K[j_W,s]^\vee}(W^\vee)\right)^{n+1-p,q}. \end{equation}
\qed
\end{pro}
The unprojected state space projects to the sum of state spaces of the form 
$\Hcal_{W,H,g}$
after taking $j_W$-invariant elements.
\begin{pro}\label{pro:jinv}
We have 
$$\HH^{j_W}=\bigoplus_{b=0}^{k-1} \Hcal_{W,K[j_W],s^b}.$$
In particular, if $W$ is Calabi--Yau, we have 
$$\HH^{j_W}=\bigoplus_{b=0}^{k-1} H_{s^b}^*(\Sigma_{W,K[j_W]};\CC),$$
where $(\frac{b}k+p,\frac{b}k+q)$-classes in $H_{s^b}^*(\Sigma_{W,K[j_W]};\CC)$
match $(\frac{b}k+p,\frac{b}k+q)$-classes in $\Hcal_{W,K[j_W],s^b}$
for any $p,q\in \ZZ$.  
\qed
\end{pro}

\subsection{The twist and the elevators}
Throughout this section the polynomial $W$ and the group $K$ will be fixed; 
 we simplify the notation  and write $$\HH:=\Hbar_{K[j_W,s]}^K(W),\qquad 
 j:=j_W,\qquad \HH^\vee:=\Hbar_{K^\vee}^{K[j_W,s]^\vee}(W^\vee),\qquad j^\vee=j_{W^\vee}.$$
 We also write $M$ for the mirror map $M_W$.

Note that the monomial element 
$(\phi,h)\in \HH$ with $$\phi=\prod_i x_i^{b_i}\bigwedge_{i\in I} dx_i\qquad \text{and}\qquad  
I=\{i\mid h\cdot x_i=x_i\}$$ is an 
eigenvector with respect to the diagonal symmetry
$\al=[p_0,\dots,p_n]$: the 
eigenvalue is 
$\exp(2\pi\cxi  \sum_j b_jp_j).$
It is natural to 
attach to each $(\phi,g)$ and $\al$ the 
so-called 
$\al$-{charge} of the form $\phi$ defined on the
$g$-fixed space: 
$$Q_\al\colon (\phi,g)\mapsto Q_\al(\phi,g)=
\sum_{j\in J} b_jp_j \mod \ZZ.$$

We decompose $\Hbar$ as  
$$\Hbar= \bigoplus_{a=0}^{k-1} \bigoplus_{b=0}^{k-1}
\bigoplus_{g\in j^{a}s^{b} K} \left(\Jac W_g\right)^K.$$
 and we can consider the following
$\QQ/\ZZ$-valued gradings on 
the set of generators 
$$\left(\phi=\prod_{i\in I} x_i^{b_i-1}\bigwedge_{i\in I} dx_i,\quad g=[p_0,p_1,\dots,p_n]\in j^as^bK\right)$$
\begin{enumerate}
\item the $j$-charge
$Q_j=Q_j\colon (\phi,g)\mapsto Q_j(\phi,g);$
\item the $j$-degree 
$d_j=\frac{a}k;$
\item the $s$-charge 
$Q_s=Q_s\colon (\phi,g)\mapsto Q_s(\phi,g);$
\item the $s$-degree
$d_s=\frac{b}k.$
\end{enumerate}
\begin{figure}
\begin{picture}(200,150)(-40,0)
  \put(0,40){\line(3,-1){69}}
  \put(69,17){\line(3,1){69}}
  \put(69,17){\line(0,1){63}}
  \put(69,80){\line(3,1){69}}
  \put(69,80){\line(-3,1){69}}
  \put(138,40){\line(0,1){63}}
  \put(0,40){\line(0,1){63}}
  \put(138,103){\line(-3,1){69}}
  \put(0,103){\line(3,1){69}}
  \put(0,103){\line(2,-3){54}}
  \put(69,126){\line(2,-3){54}}
  \put(54,22){\line(3,1){69}}
  \put(159,10){\line(0,1){120}}
      \put(152,110){\footnotesize
$\tau$}
  \put(159,107){\line(-1,0){45}}
  \put(114,70){\vector(0,1){10}}
  \put(114,70){\vector(0,-1){5}}
  \put(274,70){\vector(0,1){10}}
  \put(274,70){\vector(0,-1){5}}
    \put(116,70){\footnotesize
$e^m$}
    \put(276,70){\footnotesize
$e^f$}
  \put(159,107){\vector(1,0){135}}
    \put(149,62){\footnotesize
$M$}
  \put(159,60){\vector(-1,0){15}}
  \put(159,60){\vector(1,0){15}}
  \put(30,130){\footnotesize
\text{moving}}
  \put(210,130){\footnotesize
\text{fixed}}
  \put(55,60){\footnotesize
$Q_s$}
  \put(17,60){\footnotesize
$D$}
  \put(295,60){\footnotesize
$D'$}
  \put(215,60){\footnotesize
$d_j-d_s$}
  \put(0,10){\footnotesize
$Q_s-Q_j$}
  \put(180,10){\footnotesize
$(Q_s)-Q_j$}
  \put(100,10){\footnotesize
$d_j(=-d_s)$}
  \put(280,10){\footnotesize
$d_j$}
  \put(180,40){\line(3,-1){69}}
  \put(249,17){\line(3,1){69}}
  \put(249,17){\line(0,1){63}}
  \put(249,80){\line(3,1){69}}
  \put(249,80){\line(-3,1){69}}
  \put(318,40){\line(0,1){63}}
  \put(180,40){\line(0,1){63}}
  \put(318,103){\line(-3,1){69}}
  \put(318,103){\line(-2,-3){54}}
   \put(180,103){\line(3,1){69}}
%  \qbezier(15,5)(45,20)(45,50)
%  \qbezier(35,5)(5,20)(5,50)
%  \qbezier(5,50)(5,70)(25,70)
%  \qbezier(25,70)(45,70)(45,50)
%  \put(25,11){\circle*{8}}
%  \put(95,11){\circle*{8}}
%  \put(315,11){\circle*{8}}
%  \put(58,30){\footnotesize
%$\sqcup$}
%  \put(53,50){\footnotesize
%$\longleftrightarrow$}
%  \qbezier(85,5)(115,20)(115,50)
%  \qbezier(105,5)(75,20)(75,50)
%  \qbezier(75,50)(75,70)(95,70)
%  \qbezier(95,70)(115,70)(115,50)
%   \put(190,40){\vector(1,0){30}}
%  \qbezier(305,5)(335,20)(335,50)
%  \qbezier(325,5)(295,20)(295,50)
%  \qbezier(295,50)(295,70)(315,70)
%  \qbezier(315,70)(335,70)(335,50)
 \end{picture}
 
 \caption{two blocks (with elevators) representing the coordinates of the moving subspace $\HH^m$ and of the fixed subspace $\HH^f$. The condition $Q_j=0$ defines a plane cutting the diagonal $D$ of the left hand side face of the moving block; $D$ represents the moving part of $H^*_{\id}(\Sigma_{W,H};\CC)$. On the fixed block, the same condition $Q_j=0$ defines the face on the right hand side; within it, the diagonal $D'$ is symmetrical to $D$ and represents the fixed part of 
$H^*_{\id}(\Sigma_{W,H};\CC)$.}\label{figure}
\end{figure}
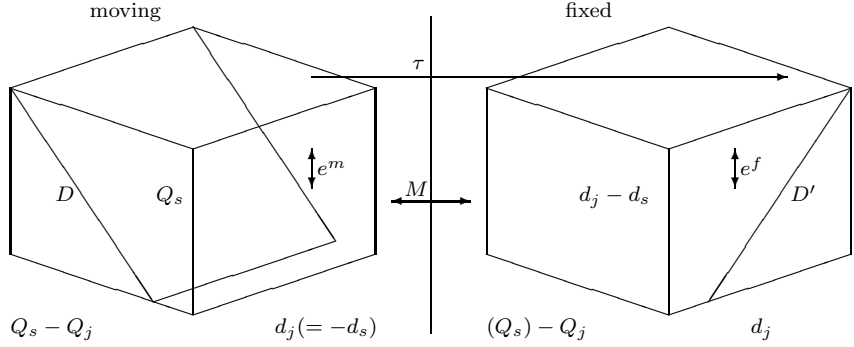
We can now decompose 
$$\HH=\bigoplus_{{0\le a,b,c,d\le k-1}}
\left[\Hbar\mid (d_j,d_s,Q_j,Q_s)=\frac1k(a,b,c,d)\right].$$
The following proposition further simplifies the decomposition.
\begin{pro}[the moving subspace, the fixed subspace)]\label{rem:support}
For any element $(\phi,g)$ we have either 
(i) $d_s=-d_j$, or (ii) $Q_s \neq 0$.
\end{pro}
\begin{proof}
This happens because $Q_s= 0$ if and only if 
$g\cdot x_0=x_0$, \emph{i.e.} 
$0\in I$. By definition of $d_s$ and $d_j$ we have 
$g\in j^{kd_j}s^{kd_s}K$ and $g\cdot x_0=\exp(2\pi\cxi (d_j+d_s)) x_0$. 
We conclude that $Q_s=0$ if and only if 
$d_s+d_j\in \ZZ$.
\end{proof}
In other words $\HH$ decomposes into an $s$-moving part $\HH^m$ ($Q_s\neq0$)
and an $s$-fixed part $\HH^f$ ($Q_s=0$)
$$\HH=\HH^m\oplus \HH^f=[\HH\mid Q_s\neq 0]\oplus [\HH\mid Q_s=0]$$
and the first summand  is $[\HH\mid d_j+d_s=0]$; hence 
the three parameters $d_j,Q_j,Q_s$ suffice for decomposing 
$\HH^m$ and 
the three parameters $d_j,d_s,Q_j$ suffice for decomposing $\HH_F$.
We write
\begin{align*}\HH^m&=
\bigoplus_{\substack{{0\le X,Y<k}\\{0< Z< k}}}
\left[\Hbar\mid (d_j,Q_s-Q_j,Q_s)=\frac1k(X,Y,Z)\right]=\bigoplus_{\substack{{0\le X,Y< k}\\{0< Z<k}}}\HH^m_{X,Y,Z}
,\\ 
\HH^f&=
\bigoplus_{\substack{{0\le X,Y< k}\\{0< Z< k}}}
\left[\Hbar\mid (d_j,-Q_j,d_j+d_s)=\frac1k(X,Y,Z)\right]
=\bigoplus_{\substack{{0\le X,Y< k}\\{0< z< k}}}\HH^f_{X,Y,Z},\end{align*}
where the choice of the three parameters in $\{0,\dots, k-1\}$ modulo $k\ZZ$
\begin{equation}\label{eq:coord}
X=kd_j,\ \ Y=\begin{cases}k(Q_s-Q_j)=-kQ_j& {\text{in }\HH^m}\\ k(Q_s-Q_j) & {\text{in }\HH^f}\end{cases}, \ \ Z
=\begin{cases}kQ_s& {\text{in }\HH^m}\\ k(d_j+d_s) & {\text{in }\HH^f}\end{cases}\end{equation}
is motivated by 
the following fact.
\begin{pro}[twist]\label{pro:twist}
For $Z=1,\dots, k-1$, we have an isomorphism
\begin{align*}\tau\colon \HH^m_{X,Y,Z}&\to \HH^f_{X,Y,Z}\\
(x_0^{Z-1}dx_0\wedge\phi,g)&\mapsto (\phi,s^{Z}g)
\end{align*}
transforming $(p,q)$-classes into  $(p-1+2Z/k,q)$-classes.
\end{pro} 
\begin{proof}
Indeed the above homomorphism  exchanges $d_j+d_s$ with $Q_s$  
and 
preserves $d_j$ and $Q_s-Q_j$.
\end{proof}

\begin{rem}
The index $Q_s-Q_j$ may be regarded as the (opposite of) the $j$-charge
of the form $\phi$ restricted to $(x_0=0)$. 
\end{rem}

There are natural isomorphisms matching $\HH^f_{X,Y,1}\cong 
\HH^f_{X,Y,2}\cong\dots\cong \HH^f_{X,Y,k-1}$ and 
$\HH^m_{X,Y,1}\cong 
\HH^m_{X,Y,2}\cong\dots\cong \HH^m_{X,Y,k-1}$. We 
refer to them as ``elevators''.
\begin{pro}[elevators] \label{pro:elevators}
For any $0\le X,Y< k$ and $0<Z'<Z''<k$
we have the isomorphisms
\begin{align*}
e^m_{Z',Z''}\colon\HH^f_{X,Y,Z'}&\to \HH^f_{X,Y,Z''}&&&&&e^f_{Z',Z''}\colon\HH^f_{X,Y,Z'}&\to \HH^f_{X,Y,Z''}\\
(\phi, g)&\mapsto (x_0^{Z''-Z'}\phi, g)&&&&&(\phi, g)&\mapsto (\phi, s^{Z''-Z'}g),
\end{align*}
with $e^m_{Z',Z''}$ and $e^f_{Z',Z''}$ 
transforming $(p,q)$-classes into classes whose bidegrees equal
 $(p-(Z''-Z')/k,q+(Z''-Z')/k)$ and $(p+(Z''-Z')/k,q+(Z''-Z')/k)$, respectively.\qed
 \end{pro}
For $0<Z'<Z''<k$, we set $e^m_{Z'',Z'}=(e^m_{Z',Z''})^{-1}$ and $e^f_{Z'',Z'}=(e^f_{Z',Z''})^{-1}$.
  
Proposition \ref{pro:MS} specializes to the following statement. 
 \begin{pro}\label{pro:mirrorcharges}
 The mirror isomorphism $M$ yields isomorphisms
 \begin{align*}
M\colon \HH^f_{X,Y,Z}\xrightarrow{\sim} (\HH^\vee)^m_{Y,X,Z},\qquad \qquad M\colon \HH^m_{X,Y,Z}&\xrightarrow{\sim} (\HH^\vee)^f_{Y,X,Z}.
\end{align*}
 \end{pro}
 
 \begin{proof}
 Let us consider $M$ as a morphism mapping $U_W$ to $U_{W^\vee}$. 
 For $W=(x_0)^k+f$ we have $W^\vee=(\ol x_0)^k+f^\vee$. Using \eqref{eq:doubledecomp},
 every state of the form 
 $(\phi,g)\in U_W$ can be regarded as an element of 
 $$(\phi,g)\in U^{b_1}_{a_1}\otimes U^{b_2}_{a_2}$$
 with $a_1\in \Aut_{(x_0)^k}=\ZZ/k$, $b_1\in \Aut_{(\ol x_0)^k}=\ZZ/k$, 
 $a_2\in \Aut_{f}$ and $b_2\in \Aut_{f^\vee}.$
 Example \ref{exa:xk} shows that there are only two possibilities: (1) $b_1=1$ or (2) $a_1=1$. More precisely, in case (1) 
 $(\phi,g)$ is in $\HH^f$, it is fixed by $s$, $b_1$ is the trivial symmetry $1\in \Aut_{(x_0)^k}$ and $a_1$ is the nontrivial character corresponding to $kd_s\in \ZZ/k\setminus \{0\}$. 
 In case (2) 
 $(\phi,g)$ is in $\HH^m$, it is not fixed by $s$, $a_1$ is trivial whereas $b_1$ is  the nontrivial character $kQ_s \in \ZZ/k\setminus \{0\}$. 
Since $M$ exchanges $a_1$ and $b_1$ this proves that $M$ exchanges $\HH^m$ and $\HH^f$ and 
preserves the coordinate $Z$ which coincides with $kd_s$ and $kQ_s$ within $\HH^m$ and $\HH^f$. 

Furthermore $M$ maps $U_{a_2}^{b_2}(f)$ to $U_{b_2}^{a_2}(f^\vee)$ with 
$a_2\in \SL_f[j_f]$ and $b_2\in \SL_{f^\vee}[j_{f^\vee}]$. We recall that 
$j_f^k\in \SL$ on both sides; therefore $\det a_2$ and $\det b_2$ are $\pmmu_k$-characters. 
The claim $(X,Y,Z)\mapsto (Y,X,Z)$ follows from 
$$\det a_2=-kd_j, \qquad \det b_2=-kQ_s+kQ_j,$$
where $\pmmu_k$-characters are identified with elements of $\ZZ/k$.
The first identity is immediate: $a_2$ is related to $(\phi,g)\in U_W$ by $a_2=g\!\mid_{x_0=0}$.
The identity follows from $\det (j\!\!\mid_{x_0=0})=\xi_k^{-1}$ by the Calabi--Yau condition. 
The second identity follows from the definition of
$$\ell_{a_2}\colon \Jac(f_{a_2})\to \Aut(f^\vee), \qquad \prod_j x_j^{b_j-1}\bigwedge_{j} dx_j\mapsto 
\prod_j \ol \rho_j^{b_j}$$from 
\eqref{eq:ellhmap}: 
the determinant of $\ol \rho_j$ is $\xi_d^{w_j}$; hence 
$\det(\prod_j \ol \rho_j^{b_j})$ is identified with  the $j_f$-charge $Q_{j_f}$ 
of the form $\phi$ restricted to 
$(x_0=0)$. This yields an identification between $\det b_2$ and the $\pmmu_k$-character
$kQ_j-kQ_s$.
 \end{proof}

 In view of the above proposition  
mirror symmetry operates as a plane symmetry exchanging the two blocks 
see Figure \ref{figure}.
\begin{rem}\label{rem:Fermat}
For Fermat potentials all the above discussion can be carried out more 
explicitly because the group elements $\pmb a=\prod_{i=1}^n \rho_i^{a_i}$ coincide 
with $\frac1d[a_1w_1,\dots,a_nw_n]$. By adopting this notation, the 
space $U_{a}^{b}(W)$ may be regarded as the one-dimensional space 
spanned by 
$$\left (\phi=\prod_{i=0}^n x_i^{b_i-1} \bigwedge_{i\in I} dx_i, \ \ \pmb a=\prod_{i=0}^n\rho_i^{a_i}\right)$$
where 
\begin{equation}\label{eq:equivalence_ab} 
a_i=0 \Leftrightarrow b_i\neq 0.\end{equation}
Mirror symmetry is simply an exchange of 
the $w\ZZ/d\ZZ$-valued vectors $\pmb a$ and $\pmb b$. 
The bidegree $(p,q)$ coincides with $$\left(\#(\pmb b)-\sum_{i=0}^n b_i \frac{w_i}d
+ \sum_{i=0}^n a_i \frac{w_i}d, \sum_{i=0}^n b_i \frac{w_i}d+ \sum_{i=0}^n a_i \frac{w_i}d,\right)$$
where $\#(\pmb b)$ is the number of elements $i$ such that $b_i\neq 0$.
Notice that $Q_s$ is $b_0/k$ and $d_s+d_j$ is $a_0/k$; 
therefore, the equivalence in Proposition \ref{rem:support} reads $b_0=0$ 
is a special case of \eqref{eq:equivalence_ab}. 
Furthermore we have 
$$(X,Y,Z)=\begin{cases} (a_0-|\pmb a|,b_0-|\pmb b|,b_0)&\text {on the moving side,}\\
(a_0-|\pmb a|,b_0-|\pmb b|,a_0)&\text {on the fixed side.}
\end{cases}$$
It is now clear that $M$ exchanges the moving side with the fixed side, 
$(X,Y,Z)$ with $(Y,X,Z)$, and 
$(p,q)$ with $(n+1-p,q)$.
\end{rem}

In view of Proposition \ref{pro:jinv},
we obtain the $j$-invariant contribution 
by setting $Q_j=0$. By \eqref{eq:coord}, this amounts to imposing $Y=Z$ within 
$\HH^m$ and $Y=0$ within $\HH^f$. 
We get 
$$\left(\bigoplus_{\substack{{0\le X< k}\\{0< Z< k}}}
\HH^m_{X,Z,Z} \right)\oplus
\left(\bigoplus_{\substack{{0\le X< k}\\{0< Z< k}}}
\HH^f_{X,0,Z} \right)$$

We get a picture of the $j$-invariant state space 
$\Hcal_{W,K[j_W],\id}$ 
by setting 
$X=0$ within 
$\HH^m$ and $X=Z$ within $\HH^f$
$$\Hcal_{W,K[j_W],\id}=\left(\bigoplus_{\substack{{0< t< k}}}
\HH^m_{0,t,t} \right)\oplus\left(\bigoplus_{\substack{{0<t< k}}}
\HH^f_{t,0,t}\right)$$
(we refer to Figure \ref{figure}).
More generally,
the $(d_s=b)$-part of $\HH^{j_W}$ is the 
state space $\Hcal_{W,K[j_W],s^b}$ 
(see Proposition \ref{pro:jinv}).
By \eqref{eq:coord}, we obtain it by setting 
$X=-b$ within 
$\HH^m$ and $Z=X+b$ within $\HH^f$ 
$$\Hcal_{W,K[j_W],s^b}=\left(\bigoplus_{\substack{{0< t< k}}}
\HH^m_{-b,t,t} \right)\oplus\left(\bigoplus_{\substack{{0\le t< k}\\t+b\neq 0}}
\HH^f_{t,0,t+b}\right).$$
Notice that the second summand only depends on  $\HH^f_{0,0,1}$
and $\HH^f_{1,0,1},\dots \HH^f_{k-1,0,k-1}$ since, for $b\neq 0$, it equals
\begin{equation}\label{eq:allinfo}
\Hcal_{W,K[j_W],s^b}=
\left(\bigoplus_{\substack{{0< t< k}}}
\HH^m_{-b,t,t} \right)\oplus\left(e^f_{1,b}(\HH^f_{0,0,1})\oplus 
\bigoplus_{\substack{{0< t< k}\\t+b\neq 0}}
e^f_{t,t+b}(\HH^f_{t,0,t})\right),
\end{equation}
with the convention $(e^f_{i,j})=(e^f_{j,i})^{-1}$ if $j<i$.
By Proposition \ref{pro:jinv} the above data correspond to 
$H_{s^b}^*(\Sigma_{W,K[j_W]};\CC)$ under the Calabi--Yau condition. 

 Proposition \ref{pro:crc} relates it to the cohomology
of an $s^b$-fixed locus within a crepant resolution. 
Using this geometric picture, we can predict some vanishing conditions, which we 
prove in general, without relying on any Calabi--Yau condition in the next proposition.
The first guess is immediate: since
$\langle s\rangle$ 
operates trivially on an $s$-fixed locus, it is natural to expect that 
$\Hbar^m_{-1,t,t}$ vanishes for all $t$. 
More generally, since 
$\langle s^b\rangle$ 
operates trivially on an $s^b$-fixed locus, we expect that 
$\Hbar^m_{-b,t,t}$ vanishes  if $tb\equiv 0$ mod $\ZZ$. 
We prove that this holds true regardless of any Calabi--Yau 
condition or existence of crepant resolution. 
\begin{pro}\label{pro:fixedLGdontmove}
Let $b\in \{0,\dots, k-1\}$. We have $$\Hbar^m_{b,t,t}= 0,$$
unless  
$t$ is a multiple of $k/\gcd(b,k)$ in $k\ZZ$. 
\end{pro}
\begin{proof}
We prove that $$\Hbar^m_{b,t,t}\neq 0$$ implies $bt\in k\ZZ$.
Recall that $\frac{t}k$ equals $Q_s-Q_j$. For 
$\Hbar^m_{b,t,t}\neq 0$, we can compute $Q_s-Q_j$ explicitly using 
$(\phi,j^b g)\in \Hbar^m_{b,t,t}$ with $g\in \Aut_f$ and $\phi$ 
a $g$-invariant form 
$$\phi=x_0^{kQ_s-1}\prod_{k\in I'}x^{b_k-1} dx_0\wedge\bigwedge_{k\in I'}dx_k,$$
with $I'=\{k\ge 1\mid j^b g\cdot x_k=x_k\}\subseteq \{1,\dots, N\}$. Using $w_0/d=1/k$,  we get 
$$Q_s-Q_j=Q_s-kQ_s\frac1k-\sum_{k\in I'} b_k\frac{w_k}d=-\sum_{k\in I'} b_k\frac{w_k}d.$$
Let us write $g$ as $[0,p_1\dots, p_n]\in \Aut_f$; then we have $k\in I'$ if and only if 
$$b\frac{w_k}d + p_k\in \ZZ.$$
Then $k$ divides $bt$, because  
$$b\frac{t}k = -b\sum_{k\in I} b_k\frac{w_k}d= -\sum_{k\in I} b_k b\frac{w_k}d=
\sum_{k\in I} b_k p_k\in \ZZ,$$
where the last relation holds since $\phi$ is $g$-invariant. 
\end{proof} 
\subsection{Mirror symmetry on the Landau--Ginzburg side}
In this section, we derive an interpretation of Proposition \ref{pro:mirrorcharges}
in terms of the Landau--Ginzburg state space. This amounts to 
expressing both sides of the isomorphism 
$\HH^m_{X,Y,Z}\cong \HH^f_{Y,X,Z}$ in terms of $j_W$-invariant spaces.

Consider the $j_W$-invariant  summands $$\HH^m_{X,Y,Y}\subset \HH^m \ \ \text{and}\ \  
 \HH^f_{X,0,Z}\subset \HH^f.$$ Their mirrors
are  $(\HH^\vee)^f_{Y,X,Y}$ and $(\HH^\vee)^m_{0,X,Z}$, and lie in 
the $j_W$-invariant part if and only if $X=0$ and $X=Z$. This happens if and 
only if we consider the mirror of $[\HH\mid Q_j=d_s=0]$ (imposing $X=0$ in $\HH^m$ 
and $X=Z$ in $\HH^f$ is the same as requiring $d_s=0$).

We obtain the first consequence of Proposition \ref{pro:mirrorcharges}. Let 
$$\left[\Hcal_{W,H,\id}\right]^{p,q}_{\chi_s=i}$$
be the eigenspace on which 
$s$ operates as the character $i\in \ZZ/k\ZZ$. 
 For any $H\in \Aut_W$ containing $j_W$, we have
\begin{equation}\label{eq:untwistedMS}
M\colon \left[\Hcal_{W,H,\id}\right]^{p,q}_{\chi_s=0}\xrightarrow{ \ \ \cong\ \ } \bigoplus_{i=1}^{k-1}
\left[\Hcal_{W^\vee,H^\vee,\id}\right]^{n+1-p,q}_{\chi_s=i},
\end{equation}
where $H=K[j]$. 

We now study
$\HH^f_{0,0,i}$. 
The subspace  $\HH^f_{0,0,i}$ mirrors $\HH^m_{0,0,i}$. 
By applying the twist $\tau$ from Proposition \ref{pro:twist}, 
we land again on $\HH^f_{0,0,i}$
which is a part of the $j$-invariant state space $\Hcal_{W^\vee,H^\vee,s^i}$. 
 
Using the elevators of Proposition \ref{pro:elevators}, we obtain 
a homomorphism 
\begin{equation} \label{eq:elevator-restriction}
el_i:=\bigoplus_{t\neq 0,k-i} e^f_{t,t+i} \colon \bigoplus_{\substack{{t \neq 0, k-i}}}
\HH^f_{t,0,t} \xrightarrow{\ \ \ \ \ \ } 
\left[\Hcal_{ W,H,s^i}\right]^s\end{equation}
whose cokernel coincides with $\HH_{0,0,i}^f$. Notice that $t+i$ is understood to be mod $k$. This means that the effect of the map on grading can be described as:
\[\im(el_i)(\tfrac{i}{k})\cong \bigoplus_{0<t<k-i} \HH^f_{t,0,t} \oplus \bigoplus_{k-i<t<k} \HH^f_{t,0,t}(1).\]
We write $el^\vee_i$ for the same construction on the mirror. We conclude 
\begin{equation}
M\colon \left[\frac{\Hcal_{W,H,s^i}(\tfrac{i}{k})^s}{\im(el_i)(\tfrac{i}{k})}\right]^{p,q}
\xrightarrow{\ \ \cong\ \ }
\left[\frac{\Hcal_{W^\vee,H^\vee,s^i}(\tfrac{i}{k})^s}{\im(el^\vee_i)(\tfrac{i}{k})}\right]^{n-p,q},
\end{equation}
where the bidegrees have been computed using $\Hcal(\frac{i}{k})^{p,q}=\Hcal^{p+i/k,q+i/k}$, the fact that $M_W$ transforms $(p,q)$-classes to $(n+1-p,q)$-classes, 
and the twist $\tau$ maps $(p,q)$-classes to $(p-1+\frac{2i}{k},q)$-classes.

Note that using \eqref{eq:untwistedMS} (recalling that $d_j$ and $d_s$ switch under mirror symmetry)  we can write 
\[[\im(el^\vee_i)(\tfrac{i}{k})]^{n-p,q}=
\bigoplus_{0<j<k-i} \left[\Hcal_{W,H,\id}(1,0)\right]^{p,q}_{\chi_s=j}\oplus \bigoplus_{k-i<j<k} \left[\Hcal_{W,H,\id}(0,1)\right]^{p,q}_{\chi_s=j}.\]
Write $\Hcal_{g}$ for $\Hcal_{W,H,g}$ and $\Hcal^\vee_g$ for $\Hcal_{W^\vee,H^\vee,g}.$  We obtain
\begin{equation}\begin{split}\label{eq:stwistedMS}
M\colon [\Hcal_{s^i}\left(\tfrac{i}{k}\right)]_{\chi_s=0}^{p,q} \oplus \bigoplus_{j<k-i} \left[\Hcal_{\id}(1,0)\right]^{p,q}_{\chi_s=j}\oplus \bigoplus_{j>k-i} \left[\Hcal_{\id}(0,1)\right]^{p,q}_{\chi_s=j} 
\xrightarrow{\ \ \cong\ \ } \\
[\Hcal^\vee_{s^i}\left(\tfrac{i}{k}\right)]_{\chi_s=0}^{n-p,q}
 \oplus \bigoplus_{j<k-i} \left[\Hcal^\vee_{\id}(1,0)\right]^{n-p,q}_{\chi_s=j}\oplus \bigoplus_{j>k-i} \left[\Hcal^\vee_{\id}(0,1)\right]^{n-p,q}_{\chi_s=j},
\end{split}\end{equation}
where $0<j<k.$
Finally we focus on the moving part of 
$\Hcal_{W,K[j_W],s^b}$ which, by \eqref{eq:allinfo} can be written 
as 
$\bigoplus_{0< t< k}
\HH^m_{k-b,t,t}$. This is the 
decomposition of $\Hcal_{W,K[j],s^b}$
into eigenspaces corresponding to the $s$-action operating 
as the character $t\in \ZZ/k$. 
By applying the mirror map $M_W$, the twist $\tau^{-1}$,
and the elevator $e^f_{t,k-b}$
 we get 
$$\HH^m_{k-b,t,t}\xrightarrow{\ M_W\ } 
(\HH^\vee)^f_{t,k-b,t}\xrightarrow{\ \tau^{-1}\ }
(\HH^\vee)^m_{t,k-b,t}\xrightarrow{\ e^m_{t,k-b}\ }
(\HH^\vee)^m_{t,k-b,k-b}.$$
Therefore we have
\begin{equation}\label{eq:movingMS}
\left[\Hcal_{W,H,s^b}\left(\textstyle\frac{b}{k}\right)\right]^{p,q}_{\chi_s=t}\cong 
\left[\Hcal_{W^\vee,H^\vee,s^t}\left(\textstyle\frac{k-t}{k}\right)\right]^{n-p,q}_{\chi_s=k-b}.\end{equation}
Notice that the map on the bidegrees is the composite of \begin{enumerate}
\item a shift
$(p,q)\mapsto (p+b/k,q+b/k)$, 
\item mirror symmetry  
$(p,q)\mapsto (n+1-p,q)$, \item 
$\tau^{-1}$ yielding 
$(p,q)\mapsto (p+1-2t/k,q)$, 
\item the elevator 
yielding $(p,q)\mapsto (p+(b-(k-t))/k,q+(b+(k-t))/k)$
\item a shift backwards $(p,q)\mapsto (p-(k-t)/k,q-(k-t)/k)$,
\end{enumerate}
inducing $(p,q)\mapsto (n-p,q)$.

In the following statement we apply to \eqref{eq:untwistedMS}, 
\eqref{eq:stwistedMS}, and \eqref{eq:movingMS} to the 
geometric interpretation \eqref{eq:statespace_alg_geom} 
of the Landau--Ginzburg state space 
in terms of relative cohomology of $(\VV_H,\FF_{W,H})$ provided in 
\S\ref{sect:Jacobiring}. 

Let $W=x_0^k+f(x_1,\dots,x_n)$ be a
 quasi-homogeneous non-degenerate polynomial 
 of degree $d$ and weights $w_1,\dots,w_n$. 
Assume $j_W\in H\subseteq \SL_W$ 
(in particular $\sum_j w_j$ is a positive multiple of $d\NN$).
Then $W$ descends to 
$\VV_{H}=[\VV/H_0]\to \CC$ and its generic fibre is $\FF_{W,H}$.
Consider the automorphism  $s=[\frac1k,0,\dots,0]\colon \VV_H\to \VV_H$,
the orbifold cohomology groups $H_{\id}^*(\VV_H,\FF_{W,H})$
and $H_{s}^*(\VV_H,\FF_{W,H})$.

For $0<i<k$, define $\overline{H}^*_{\id,i}(\VV_{H},\FF_{W,H})$ to be the bigraded vector space
\[\bigoplus_{j<k-i}\left[H^*_{\id}(\VV_{H},\FF_{W,H})\left(1,0\right)\right]_{\chi_s=j} \oplus \bigoplus_{j>k-i}\left[H^*_{\id}(\VV_{H},\FF_{W,H})\left(0,1\right)\right]_{\chi_s=j};\]
here $j \in \{1,\dots,k-1\}.$
This is the padding needed to state the mirror theorem. 

\begin{thm}[mirror theorem for Landau--Ginzburg models]\label{thm:MS_LG}
Let $W=x_0^k+f(x_1,\dots,x_n)$ be a
 quasi-homogeneous, non-degenerate, invertible polynomial and $H$ a group of symmetries 
 satisfying $j_W\in H\subseteq \SL_W$. 
As above, the polynomial $W$ descends to 
$\VV_{H}=[\VV/H_0]\to \CC$ and its generic fibre is $\FF_{W,H}$.

Then, for $b$ and $t\neq 0$, 
we have 
\begin{enumerate}
\item $H^{p,q}_{\id}(\VV_{H},\FF_{W,H})_{\chi_s=0}\cong 
\bigoplus_{i=1}^{k-1} H^{n+1-p,q}_{\id}(\VV_{H^\vee},\FF_{W^\vee,H^\vee})_{\chi_s=i};$
\item Let $\FF:=\FF_{W,H}$ and $\FF^\vee:=\FF_{W^\vee,H^\vee}$. 
 
 For $0<i<k$,  
%\begin{multline*}
\[ \left[H^{p,q}_{s^i}(\VV_H,\FF)\left(\tfrac{i}{k} \right)\right]^s \oplus \overline{H}^{p,q}_{\id,i}(\VV_H,\FF) \cong
\left[H^{n-p,q}_{s^i}(\VV_{H^\vee},\FF^\vee)\left(\tfrac{i}{k}\right)\right]^s \oplus \overline{H}_{\id,i}^{n-p,q}(\VV_{H^\vee},\FF^\vee) ;\]
%\end{multline*}

\item $H^{p,q}_{s^b}(\VV_{H},\FF_{W,H})\left(\frac{b}k\right)_{\chi_s=t}\cong 
H^{n-p,q}_{s^{-t}}(\VV_{H^\vee},\FF_{W^\vee,H^\vee})
\left(\frac{k-t}k\right)_{\chi_s=-b}.$
\end{enumerate}
\end{thm}
\begin{proof}
Since $H$ equals $K[j_W]$ for a suitable $K\subseteq \SL_f$ containing $j_W^k$,
we can conclude via $\Hcal _{P,H,g}^{p,q} =H_{g}^{p,q}(\VV_H;\FF_{P,H})$. 
\end{proof}
\section{Geometric mirror symmetry}\label{sect:geometry}
If $W$ is of Calabi--Yau type,
via the Landau--Ginzburg/Calabi--Yau correspondence  of Theorem \ref{thm:lgcy}
based on $\Phi\colon H^*(\VV;\CC)\to H^*(\LL;\CC)$, we 
provide an equivalent statement on the Calabi--Yau side. 

The existence of the isomorphism $\Phi$ is guaranteed by the Calabi--Yau condition 
(ensuring $K$-equivalence). As before, for $0<i<k$, define $\overline{H}^*_{\id,i}(\Sigma_{W,H})$ to be the bigraded vector space
\begin{equation}\label{eq:movingcycles}\bigoplus_{j<k-i}\left[H^*_{\id}(\Sigma_{W,H})\left(1,0\right)\right]_{\chi_s=j} \oplus \bigoplus_{j>k-i}\left[H^*_{\id}(\Sigma_{W,H})\left(0,1\right)\right]_{\chi_s=j};\end{equation}
where again, $j$ runs between $1$ and $k-1$. 
Then we have the following statement.
\bigskip
\begin{thm}[mirror theorem for CY orbifolds with automorphism $s$]\label{thm:MS_CY}
Let $W=x_0^k+f(x_1,\dots,x_n)$ be a
be a quasi-homogeneous, non-degenerate, 
 invertible, Calabi--Yau polynomial and $H$ a group of symmetries 
 satisfying $j_W\in H\subseteq \SL_W$.  Let $\Sigma=\Sigma_{W,H}$ and $\Sigma^\vee=\Sigma_{W^\vee,H^\vee}.$
Then the following holds for $b,t\neq 0$.  
\begin{enumerate}
\item Let $d=n-1$. Then $H^{p,q}_{\id}(\Sigma)_{\chi_s=0}\cong 
\bigoplus_{i=1}^{k-1} H^{d-p,q}_{\id}(\Sigma^\vee)_{\chi_s=i};$
\item Let $d=n-2$. For $0<i<k$,
\begin{equation*}
\left[H^{p,q}_{s^i}(\Sigma)\left(\tfrac{i}{k}\right)\right]^s \oplus \overline{H}^{p,q}_{\id,i}(\Sigma)\cong \left[H^{d-p,q}_{s^i}(\Sigma^\vee)\left(\tfrac{i}{k}\right)\right]^s \oplus \overline{H}_{\id,i}^{d-p,q}(\Sigma^\vee);\end{equation*}
\item Let $d=n-2$. Then $H^{p,q}_{s^b}(\Sigma)_{\chi_s=t}\left(\frac{b}k\right)\cong 
H^{d-p,q}_{s^{-t}}(\Sigma^\vee)\left(\frac{k-t}k\right)_{\chi_s=-b}.$
\end{enumerate}
\end{thm}

\begin{proof} This follows immediately from Theorem \ref{thm:MS_LG} and the LG/CY correspondence. 
\end{proof}
\begin{rem} In the theorem, $d$ denotes the maximum of the dimensions of the components of the inertia stack considered in each case. 
\end{rem}
For $k=2$, the second equation of the statement of Theorem \ref{thm:MS_CY} can be stated as a mirror symmetry statement involving the cohomology 
groups $H^{*}_s$. 
Notice that the first statement says that Berglund--H\"ubsch mirror symmetry exchanges 
invariant  $(p,q)$-classes for $\Sigma_{W,H}$ and  anti-invariant
$(n-1-p,q)$-classes of $\Sigma_{W^\vee,H^\vee}$ (and vice versa).
Finally the third statement is trivial because both sides vanish by Proposition \ref{pro:fixedLGdontmove}. 
In this way we recover the main theorem of \cite{CKV}.
\begin{cor}
Let $W=x_0^2+f(x_1,\dots,x_n)$ be a
be a quasi-homogeneous, non-degenerate, 
 invertible, Calabi--Yau polynomial and $H$ a group of symmetries 
 satisfying $j_W\in H\subseteq \SL_W$. 
Then,
we have 
\begin{align*}
H^{p.q}_{\id}(\Sigma_{W,H})^{\pm}&\cong H^{n-1-p,q}_{\id}(\Sigma_{W^\vee,H^\vee})^\mp;\\H^{p,q}_s(\Sigma_{W,H})\textstyle{\left(\frac{1}2\right)}&\cong 
H^{n-2-p,q}_s(\Sigma_{W^\vee,H^\vee})\textstyle{\left(\frac{1}2\right)}.\end{align*}
\end{cor}

\begin{exa}\label{ex:elliptic}
Let us consider  $E=(x^6+y^3+z^2=0)$ 
within $\PP(1,2,3)$ with its order-$6$ symmetry $s=[\frac16,0,0]$. 
In this case the ``Calabi--Yau orbifold'' is represented by an elliptic curve. 
The cohomology groups $H_{s^b}^*$ describe the cohomology of the 
$s^b$-fixed loci $E_b$, shifted by $(\frac{b}6,\frac{b}6)$.
Furthermore, the mirror of $E$ coincides with $E$,  because the defining equation is of
Fermat type and $J$ equals $\SL$ (the order of $\SL$ is $w_xw_yw_z/\deg$ and equals the order $\deg$ of $J$). 
This example allows us to test $\HH$ as a state space 
computing the cohomology of $E$, and the cohomology of its fixed spaces 
satisfying $E_1=E_{2}\cap E_{3}$ and $E_{2}=E_{4}$. 
Since $E$ is the elliptic curve with order-$6$ complex multiplication, $E_1$ is the origin and the fixed spaces  $E_{3}(=E[2])$ and $E_{2}$ are respectively a set of $4$ points and 3 points interecting at the origin. Clearly $E_{3}\setminus E_1$ is the unique 
order-$3$ orbit and 
$E_{2}\setminus E_1$ is the unique order-$2$ orbit.

The $b$th  row in the table below represents the ranks of contributions of $\HH[d_s=\frac{b}6]$ whereas the 
$a$th column represents the contributions to the state space of $\HH[d_j=\frac{a}6]$.
Notice that, by means of the elevators, all rows are identical except for the anti-diagonal 
entries of the form $\HH[d_j+d_s=0]$, which we underlined. 
\begin{center}
\begin{tabular}{|c|c|c|c|c|c|c|}
\hline
$\dim(H_{\id})=$& \ul{2} & 1 & 0 & 0 & 0 & 1 \\
\hline
$\dim(H_s)=$&0 & 1 & 0 & 0 & 0 & \ul{0} \\
\hline
$\dim(H_{s^2})=$&0 & 1 & 0 & 0 & \ul{1} & 1 \\
\hline
$\dim(H_{s^3})=$&0 & 1 & 0 & \ul{2} & 0 & 1 \\
\hline
$\dim(H_{s^4})=$&0 & 1 & \ul{1} & 0 & 0 & 1 \\
\hline
$\dim(H_{s^5})=$&0 & \ul{0} & 0 & 0 & 0 & 1 \\
\hline
\end{tabular}
\end{center}
The $0$th row is the $4$-dimensional cohomology of the elliptic curve $E$ organised in its 
$2$-dimensional primitive part (spanned by the forms $dx\wedge dy\wedge dz$
and $x^4ydx\wedge dy\wedge dz$)
and its $2$-dimensional ambient part arising in the state space $\HH[d_s=0, d_j=a/6]$ for $a=1$ 
and $a=5$
($j$ and $j^5$ correspond to the only narrow sectors of the state space, \emph{i.e.} the 
only powers of $[\frac16,\frac13,\frac12]$ fixing only the origin). 
On the row corresponding to $H^*_s$ there is a single contribution for $d_j=1/5$. This happens because $E_1$ is a point. Furthermore 
$\HH[d_j=5/6, d_s=1/6]=\HH[d_j=1/6, d_s=5/6]$ vanish 
by Proposition \ref{pro:fixedLGdontmove}. 
The remaining anti-diagonal terms are 
$\HH[d_j=4/6, d_s=2/6]=\HH[d_j=2/6, d_s=4/6]=
\langle x^2dx\wedge dz\rangle$ and 
$\HH[d_j=3/3, d_s=3/3]=\langle xydx\wedge dy,x^3dx\wedge dy\rangle$.

The above mirror symmetry statement (1) involves the first row and claims that 
all fixed cohomology classes appearing for $d_j=\frac16,\dots,\frac56$ match 
the classes of $\HH[d_j=0, d_s=0]$; we already noticed that this identifies two $2$-dimensional spaces of ambient and primitive cohomology. 
Statement (2), for $i=1$, says that the 1 dimensional space
$\HH[d_s=0, q_j=1,2,3,4]$ (spanned by the class $dx\wedge dy\wedge dz$) matches the cohomology class spanned by $x^4 dx \wedge dy \wedge dz$.

Statement (3) is a map $M\colon xydx\wedge dy\mapsto x^2dx\wedge dz\in \HH[d_j=\frac26, d_s=\frac46]$
and a map $M\colon x^3dx\wedge dy\mapsto x^2dx\wedge dz\in \HH[d_j=\frac46, d_s=\frac26]$.
In this way $$M\colon \HH\left[d_j=\frac36, d_s=\frac36\right]\xrightarrow {\ \ \ \cong\ \ \ }\HH\left[d_j=\frac26, d_s=\frac46\right]\oplus \HH\left[d_j=\frac46, d_s=\frac26\right].$$
In geometric terms mirror symmetry matches the order-$2$ orbit to the order-$3$ orbit. More 
precisely, the mirror statement (3) claims that there are as many eigenvectors 
of eigenvalue $(\chi_6)^2$ and $(\chi_6)^4$ in the cohomology of $E_{3}$ as 
eigenvectors 
of eigenvalue $(\chi_6)^3$ in the cohomology of $E_{2}$ and of 
$E_{4}$.
\end{exa}

\begin{exa}
We consider the genus-$3$ curve $C$ defined by the degree-$4$ Fermat quartic 
$x_1^4+x_2^4+x_3^4=0$ in $\PP^2$. The $4$-fold cover of $\PP^2$ ramified on $C$ 
is a K3 surface defined as the vanishing locus of the polynomial $W=x_0^4+x_1^4+x_2^4+x_3^4$. In this example, the Calabi--Yau orbifold $\Sigma_{W}$  is again representable and we can treat the cohomologies $H_{\id}^*$ and $H_s^*$ as ordinary 
cohomologies of the K3 surface and of the ramification locus. As in the previous example, we display the cohomological data in a table. 
The $b$th  row in the table below represents the ranks of contributions of $\HH[d_s=\frac{b}4]$ whereas the 
$a$th column represents the contributions to the state space of $\HH[d_j=\frac{a}4]$.

\begin{center}
\begin{tabular}{|c|c|c|c|c|} 
\hline
$\dim(H_{\id})=$& \begin{tikzpicture}[scale=0.4] \tiny
\node at (1,2) {0};
\node at (1,1) {\color{red}{6}\color{black}{+}\color{blue}{7}\color{black}{+}\color{green}{6}};
\node at (-0.5,1) {\color{red}{1}};
\node at (2.5,1) {\color{green}{1}};
\node at (1,0) {0};
\end{tikzpicture} & \begin{tikzpicture}[scale=0.4]\tiny
\node at (1,2) {0};
\node at (1,1) {0};
\node at (0,1) {0};
\node at (2,1) {0};
\node at (1,0) {1};
\end{tikzpicture} & 
\begin{tikzpicture}[scale=0.4]\tiny
\node at (1,2) {0};
\node at (1,1) {1};
\node at (0,1) {0};
\node at (2,1) {0};
\node at (1,0) {0};
\end{tikzpicture} & \begin{tikzpicture}[scale=0.4]\tiny
\node at (1,2) {1};
\node at (1,1) {0};
\node at (0,1) {0};
\node at (2,1) {0};
\node at (1,0) {0};
\end{tikzpicture}  \\
\hline
$\dim(H_s)=$& \begin{tikzpicture}[scale=0.4]\tiny
\node at (0,0) {0};
\node at (0,1) {0};
\node at (-0.5,0.5) {3};
\node at (0.5,0.5) {3};
\end{tikzpicture} & \begin{tikzpicture}[scale=0.4]\tiny
\node at (0,0) {1};
\node at (0,1) {0};
\node at (-0.5,0.5) {0};
\node at (0.5,0.5) {0};
\end{tikzpicture} & \begin{tikzpicture}[scale=0.4]\tiny
\node at (0,0) {0};
\node at (0,1) {1};
\node at (-0.5,0.5) {0};
\node at (0.5,0.5) {0};
\end{tikzpicture} & \begin{tikzpicture}[scale=0.4]\tiny
\node at (0,0) {0};
\node at (0,1) {0};
\node at (-0.5,0.5) {0};
\node at (0.5,0.5) {0}; \end{tikzpicture} \\
\hline
$\dim(H_{s^2})=$& \begin{tikzpicture}[scale=0.4]\tiny
\node at (0,0) {0};
\node at (0,1) {0};
\node at (-0.5,0.5) {3};
\node at (0.5,0.5) {3};
\end{tikzpicture} & \begin{tikzpicture}[scale=0.4]\tiny
\node at (0,0) {1};
\node at (0,1) {0};
\node at (-0.5,0.5) {0};
\node at (0.5,0.5) {0};
\end{tikzpicture} & \begin{tikzpicture}[scale=0.4]\tiny
\node at (0,0) {0};
\node at (0,1) {0};
\node at (-0.5,0.5) {0};
\node at (0.5,0.5) {0};
\end{tikzpicture} & \begin{tikzpicture}[scale=0.4]\tiny
\node at (0,0) {0};
\node at (0,1) {1};
\node at (-0.5,0.5) {0};
\node at (0.5,0.5) {0}; \end{tikzpicture} \\
\hline
$\dim(H_{s^3})=$& \begin{tikzpicture}[scale=0.4]\tiny
\node at (0,0) {0};
\node at (0,1) {0};
\node at (-0.5,0.5) {3};
\node at (0.5,0.5) {3};
\end{tikzpicture} & \begin{tikzpicture}[scale=0.4]\tiny
\node at (0,0) {1};
\node at (0,1) {0};
\node at (-0.5,0.5) {0};
\node at (0.5,0.5) {0};
\end{tikzpicture} & \begin{tikzpicture}[scale=0.4]\tiny
\node at (0,0) {1};
\node at (0,1) {0};
\node at (-0.5,0.5) {0};
\node at (0.5,0.5) {0};
\end{tikzpicture} & \begin{tikzpicture}[scale=0.4]\tiny
\node at (0,0) {0};
\node at (0,1) {1};
\node at (-0.5,0.5) {0};
\node at (0.5,0.5) {0}; \end{tikzpicture}\\
\hline

\end{tabular}
\end{center}
The colors in the table refer to the weight of $s$: cohomology in red has character $1$, blue has character $2$, and green has character $3$. 
Statement (2) involves, on one side, the cohomology of the curve $C$ ($H_s$) and the moving cohomology of the K3 surface with weights $1$ and $2$.
The total cohomology on one side of Statement (2) is thus
\[ \begin{tikzpicture}\small
\node at (0,0) {1};
\node at (0,1) {1};
\node at (-0.5,0.5) {3+1};
\node at (0.5,0.5) {3+13};
\end{tikzpicture}.\]
We notice that the only $\SL$-invariant broad cohomology 
classes in the entire unprojected state space $U(W)$ are contained in 
$U(W)_{\id}$; this implies  $U(W)_{sg}^{\SL}=0$. 
Hence $ H^*_{\mathrm{prim},s}$ vanishes. 
One can compute the mirror table as

\begin{center}
\begin{tabular}{|c|c|c|c|c|} 
\hline
$\dim(H_{\id})=$& \begin{tikzpicture}[scale=0.4] \tiny
\node at (1,2) {0};
\node at (1,1) {\color{red}{0}\color{black}{+}\color{blue}{1}\color{black}{+}\color{green}{1}};
\node at (-0.5,1) {\color{red}{1}};
\node at (2.5,1) {\color{green}{1}};
\node at (1,0) {0};
\end{tikzpicture} & \begin{tikzpicture}[scale=0.4]\tiny
\node at (1,2) {0};
\node at (1,1) {6};
\node at (0,1) {0};
\node at (2,1) {0};
\node at (1,0) {1};
\end{tikzpicture} & 
\begin{tikzpicture}[scale=0.4]\tiny
\node at (1,2) {0};
\node at (1,1) {7};
\node at (0,1) {0};
\node at (2,1) {0};
\node at (1,0) {0};
\end{tikzpicture} & \begin{tikzpicture}[scale=0.4]\tiny
\node at (1,2) {1};
\node at (1,1) {6};
\node at (0,1) {0};
\node at (2,1) {0};
\node at (1,0) {0};
\end{tikzpicture}  \\
\hline
$\dim(H_s)=$& \begin{tikzpicture}[scale=0.4]\tiny
\node at (0,0) {3};
\node at (0,1) {3};
\node at (-0.5,0.5) {0};
\node at (0.5,0.5) {0};
\end{tikzpicture} & \begin{tikzpicture}[scale=0.4]\tiny
\node at (0,0) {1};
\node at (0,1) {6};
\node at (-0.5,0.5) {0};
\node at (0.5,0.5) {0};
\end{tikzpicture} & \begin{tikzpicture}[scale=0.4]\tiny
\node at (0,0) {0};
\node at (0,1) {7};
\node at (-0.5,0.5) {0};
\node at (0.5,0.5) {0};
\end{tikzpicture} & \begin{tikzpicture}[scale=0.4]\tiny
\node at (0,0) {0};
\node at (0,1) {0};
\node at (-0.5,0.5) {0};
\node at (0.5,0.5) {0}; \end{tikzpicture} \\
\hline
$\dim(H_{s^2})=$& \begin{tikzpicture}[scale=0.4]\tiny
\node at (0,0) {3};
\node at (0,1) {3};
\node at (-0.5,0.5) {0};
\node at (0.5,0.5) {0};
\end{tikzpicture} & \begin{tikzpicture}[scale=0.4]\tiny
\node at (0,0) {1};
\node at (0,1) {6};
\node at (-0.5,0.5) {0};
\node at (0.5,0.5) {0};
\end{tikzpicture} & \begin{tikzpicture}[scale=0.4]\tiny
\node at (0,0) {0};
\node at (0,1) {0};
\node at (-0.5,0.5) {0};
\node at (0.5,0.5) {0};
\end{tikzpicture} & \begin{tikzpicture}[scale=0.4]\tiny
\node at (0,0) {6};
\node at (0,1) {1};
\node at (-0.5,0.5) {0};
\node at (0.5,0.5) {0}; \end{tikzpicture} \\
\hline
$\dim(H_{s^3})=$& \begin{tikzpicture}[scale=0.4]\tiny
\node at (0,0) {3};
\node at (0,1) {3};
\node at (-0.5,0.5) {0};
\node at (0.5,0.5) {0};
\end{tikzpicture} & \begin{tikzpicture}[scale=0.4]\tiny
\node at (0,0) {0};
\node at (0,1) {0};
\node at (-0.5,0.5) {0};
\node at (0.5,0.5) {0};
\end{tikzpicture} & \begin{tikzpicture}[scale=0.4]\tiny
\node at (0,0) {7};
\node at (0,1) {0};
\node at (-0.5,0.5) {0};
\node at (0.5,0.5) {0};
\end{tikzpicture} & \begin{tikzpicture}[scale=0.4]\tiny
\node at (0,0) {6};
\node at (0,1) {1};
\node at (-0.5,0.5) {0};
\node at (0.5,0.5) {0}; \end{tikzpicture}\\
\hline

\end{tabular}
\end{center}
From this, we see that the mirror $s$-fixed locus is four projective curves and 12 isolated fixed points. The mirror Hodge diamond for Statement (2) is
\[ \begin{tikzpicture}\small
\node at (0,0) {4};
\node at (0,1) {16};
\node at (-0.5,0.5) {0+1};
\node at (0.5,0.5) {0+1};
\end{tikzpicture}.\]
Note that despite being an order 2 automorphism, the Hodge diamonds of the fixed loci of $s^2$ do not mirror each other. 

Clearly mirror symmetry should also yield 
a relation between the quantum invariants of the primitive classes 
of the curve and the orbifold quantum invariants of these sectors. 
\end{exa}

%It is easy to generalize the above example by means of Remark \ref{rem:Fermat}. We get a statement applying to  degree-$d$ Fermat polynomials in  $m\le d$ variables. 
%\begin{cor} Let $X$ be a degree-$d$ hypersurface in $\PP^{m-1}$ with $d\ge m$.
%The mirror isomorphism matches the primitive cohomology classes of $X$ with the $\pmmu_d$-representations lying in $\SL({m};\CC)$ and
%fixing only the origin. In particular the primitive classes of bidegree $(m-1-p,p-1)$ match the age-$p$ representations.\qed\end{cor} 

The structure of the of the above example is shared by all K3 orbifolds of this type with order 4 automorphism. Combining the mirror theorem with the fact that $s$ and $s^3$ have the same fixed locus (and hence cohomology of the same dimension), we can see that for any $W=x_0^4+f(x_1,x_2,x_3)$ and group $G$, the table for $\Sigma_{W,G}$ is given by 

\begin{center}
\begin{tabular}{|c|c|c|c|c|} 
\hline
$\dim(H_{\id})=$& \begin{tikzpicture}[scale=0.4] \tiny
\node at (1,2) {0};
\node at (1,1) {\color{red}{$a$}\color{red}{$-$}\color{red}{$1$}\color{black}{+}\color{blue}{$b$}\color{black}{+}\color{green}{$a$}\color{green}{$-$}\color{green}{$1$}};
\node at (-1,1) {\color{red}{1}};
\node at (3,1) {\color{green}{1}};
\node at (1,0) {0};
\end{tikzpicture} & \begin{tikzpicture}[scale=0.4]\tiny
\node at (1,2) {0};
\node at (1,1) {$a\!^{\smallvee}\!\!-\!1$};
\node at (-0.5,1) {0};
\node at (2.5,1) {0};
\node at (1,0) {1};
\end{tikzpicture} & 
\begin{tikzpicture}[scale=0.4]\tiny
\node at (1,2) {0};
\node at (1,1) {$b\!^{\smallvee}$};
\node at (-0.5,1) {0};
\node at (2.5,1) {0};
\node at (1,0) {0};
\end{tikzpicture} & \begin{tikzpicture}[scale=0.4]\tiny
\node at (1,2) {1};
\node at (1,1) {$a\!^{\smallvee}\!\!-\!1$};
\node at (-0.5,1) {0};
\node at (2.5,1) {0};
\node at (1,0) {0};
\end{tikzpicture}  \\
\hline
$\dim(H_s)=$& \begin{tikzpicture}[scale=0.4]\tiny
\node at (0,0) {$g\!^{\smallvee}$};
\node at (0,1) {$g\!^{\smallvee}$};
\node at (-0.5,0.5) {$g$};
\node at (0.5,0.5) {$g$};
\end{tikzpicture} & \begin{tikzpicture}[scale=0.4]\tiny
\node at (0,0) {1};
\node at (0,1) {$a\!^{\smallvee}\!\!-\!1$};
\node at (-0.5,0.5) {0};
\node at (0.5,0.5) {0};
\end{tikzpicture} & \begin{tikzpicture}[scale=0.4]\tiny
\node at (0,0) {0};
\node at (0,1) {$b\!^{\smallvee}$};
\node at (-0.5,0.5) {0};
\node at (0.5,0.5) {0};
\end{tikzpicture} & \begin{tikzpicture}[scale=0.4]\tiny
\node at (0,0) {0};
\node at (0,1) {0};
\node at (-0.5,0.5) {0};
\node at (0.5,0.5) {0}; \end{tikzpicture} \\
\hline
$\dim(H_{s^2})=$& \begin{tikzpicture}[scale=0.4]\tiny
\node at (0,0) {$g\!^{\smallvee}$};
\node at (0,1) {$g\!^{\smallvee}$};
\node at (-0.5,0.5) {$g$};
\node at (0.5,0.5) {$g$};
\end{tikzpicture} & \begin{tikzpicture}[scale=0.4]\tiny
\node at (0,0) {1};
\node at (0,1) {$a\!^{\smallvee}\!\!-\!1$};
\node at (-0.5,0.5) {0};
\node at (0.5,0.5) {0};
\end{tikzpicture} & \begin{tikzpicture}[scale=0.4]\tiny
\node at (0,0) {$c$};
\node at (0,1) {$c$};
\node at (-0.5,0.5) {$c\!^{\smallvee}$};
\node at (0.5,0.5) {$c\!^{\smallvee}$};
\end{tikzpicture} & \begin{tikzpicture}[scale=0.4]\tiny
\node at (0,0) {$a\!^{\smallvee}\!\!-\!1$};
\node at (0,1) {1};
\node at (-0.5,0.5) {0};
\node at (0.5,0.5) {0}; \end{tikzpicture} \\
\hline
$\dim(H_{s^3})=$& \begin{tikzpicture}[scale=0.4]\tiny
\node at (0,0) {$g\!^{\smallvee}$};
\node at (0,1) {$g\!^{\smallvee}$};
\node at (-0.5,0.5) {$g$};
\node at (0.5,0.5) {$g$};
\end{tikzpicture} & \begin{tikzpicture}[scale=0.4]\tiny
\node at (0,0) {0};
\node at (0,1) {0};
\node at (-0.5,0.5) {0};
\node at (0.5,0.5) {0};
\end{tikzpicture} & \begin{tikzpicture}[scale=0.4]\tiny
\node at (0,0) {$b\!^{\smallvee}$};
\node at (0,1) {0};
\node at (-0.5,0.5) {0};
\node at (0.5,0.5) {0};
\end{tikzpicture} & \begin{tikzpicture}[scale=0.4]\tiny
\node at (0,0) {$a\!^{\smallvee}\!\!-\!1$};
\node at (0,1) {1};
\node at (-0.5,0.5) {0};
\node at (0.5,0.5) {0}; \end{tikzpicture}\\
\hline

\end{tabular}.
\end{center}
The table for the $\Sigma_{W^\vee,G^\vee}$ is obtained from this table by replacing $x \mapsto x^\vee.$ 

Using this table, we can find relationships between the topological invariants of the fixed loci of crepant resolutions of $\mf X:=\Sigma_{W,G}$  and its mirror $\mf X^\vee$. Example \ref{eg:dim2yasuda} shows that there is an isomorphism between the $s^{3}$-orbifold cohomology of $\mf X$ and the cohomology of the $s$ fixed locus in the resolution $\wt X$.  Recall that this is because for K3 surfaces, the age function is constant (of $3/4$) on the $s^3$-orbifold cohomology of the resolution. By similar reasoning, the $s^2$-orbifold cohomology of $\wt X$ also has a constant age function (of $1/2$). 

Now consider the following invariants for $i=1,2$:
\begin{itemize}
\item $f_i,$ the number of isolated fixed points of $s^i$;
\item $g_i,$ the sum of the genera of the fixed curves of $s^i$;
\item $N_i,$ the number of curves in the fixed locus of $s^i$. 
\end{itemize}
A superscript $\vee$ indicates the invariants of the mirror K3.  
\begin{cor}\label{cor:fourinv} We have 
\begin{enumerate}
\item $N_1=g_1^\vee+1$;
\item $N_2+g_2+f_1=20-N_2^\vee-g_2^\vee-f_1^\vee.$
\end{enumerate}
\end{cor}
\begin{proof}
The table above implies that $2a+b+2a^\vee+b^\vee=24$, and that $g_1=g,$ $N_1=g^\vee+1$, and $f_1=a^\vee+b^\vee-2$. The statements for the mirror invariants are obtained by $x \leftrightarrow x^\vee$: for example, $N_1^\vee= g+1.$ 
 
Similarly, $N_2-g_2^\vee=(g^\vee+c+a^\vee)-(g^\vee+c)=a^\vee$, which implies the statement. 
\end{proof}

The same analysis also works for K3 surfaces with prime order automorphisms. Let $W$ be a Calabi--Yau polynomial of the form $W=x_0^p+f(x_1,x_2,x_3)$ for $p$ prime. Then the Landau--Ginzburg state space breaks down as
\begin{center} \small
\begin{tabular}{|M{2cm}|M{1.5cm}|M{1.1cm}|M{1.1cm}|M{1.1cm}|M{1.1cm}|M{1.1cm}|}
\hline
$\dim(H_{\id})=$&
\begin{tikzpicture}[scale=0.4]\tiny
\node at (1,2) {0};
\node at (1,1) {$(p\!\!-\!\!1)a\!\!-\!\!2$};
\node at (-0.7,1) {1};
\node at (2.7,1) {1};
\node at (1,0) {0};
\end{tikzpicture} & 
\begin{tikzpicture}[scale=0.4]\tiny
\node at (1,2) {0};
\node at (1,1) {$a\!^{\smallvee}\!\!-\!\!\,1$};
\node at (0,1) {0};
\node at (2,1) {0};
\node at (1,0) {1};
\end{tikzpicture} & $\cdots$ &
\begin{tikzpicture}[scale=0.4]\tiny
\node at (1,2) {0};
\node at (1,1) {$a\!^{\smallvee}$};
\node at (0,1) {0};
\node at (2,1) {0};
\node at (1,0) {0};
\end{tikzpicture} & $\cdots$ &
\begin{tikzpicture}[scale=0.4]\tiny
\node at (1,2) {1};
\node at (1,1) {$a\!^{\smallvee}\!\!-\!\!\,1$};
\node at (0,1) {0};
\node at (2,1) {0};
\node at (1,0) {0};
\end{tikzpicture} \\
\hline
$\dim(H_{s})=$ &\begin{tikzpicture}[scale=0.4]\tiny
\node at (1,0) {$g\!^{\smallvee}$};
\node at (1,1) {{$g\!^{\smallvee}$}};
\node at (1.5,0.5) {$g$};
\node at (0.5,0.5) {$g$};
\end{tikzpicture} &
\begin{tikzpicture}[scale=0.4]\tiny
\node at (1,0) {1};
\node at (1,1) {$a\!^{\smallvee}\!\!-\!\!\,1$};
\node at (1.5,0.5) {0};
\node at (0.5,0.5) {0};
\end{tikzpicture} & $\cdots$ &
\begin{tikzpicture}[scale=0.4]\tiny
\node at (1,0) {0};
\node at (1,1) {$a^{\smallvee}$};
\node at (1.5,0.5) {0};
\node at (0.5,0.5) {0};
\end{tikzpicture} & $\cdots$ &
\begin{tikzpicture}[scale=0.4]\tiny
\node at (1,0) {0};
\node at (1,1) {0};
\node at (1.5,0.5) {0};
\node at (0.5,0.5) {0};
\end{tikzpicture}\\
\hline
$\dim(H_{s^2})=$&\begin{tikzpicture}[scale=0.4]\tiny
\node at (1,0) {$g\!^{\smallvee}$};
\node at (1,1) {$g\!^{\smallvee}$};
\node at (1.5,0.5) {$g$};
\node at (0.5,0.5) {$g$};
\end{tikzpicture} &
\begin{tikzpicture}[scale=0.4]\tiny
\node at (1,0) {1};
\node at (1,1) {$a\!^{\smallvee}\!\!-\!1$};
\node at (1.5,0.5) {0};
\node at (0.5,0.5) {0};
\end{tikzpicture} & 
\begin{tikzpicture}[scale=0.4]\tiny
\node at (1,0) {$0$};
\node at (1,1) {$a\!^{\smallvee}$};
\node at (1.5,0.5) {$0$};
\node at (0.5,0.5) {$0$};
\end{tikzpicture} & $\cdots$ &
\begin{tikzpicture}[scale=0.4]\tiny
\node at (1,0) {$0$};
\node at (1,1) {$0$};
\node at (1.5,0.5) {0};
\node at (0.5,0.5) {0};
\end{tikzpicture} &
\begin{tikzpicture}[scale=0.4]\tiny
\node at (1,0) {$a\!^{\smallvee}\!\!-\!1$};
\node at (1,1) {$1$};
\node at (1.5,0.5) {0};
\node at (0.5,0.5) {0};
\end{tikzpicture}\\
\hline
\multicolumn{7}{|c|}{\large $\vdots$\hspace{5cm}$\vdots$}\\
\hline
$\dim(H_{s^{p-1}})=$&\begin{tikzpicture}[scale=0.4]\tiny
\node at (1,0) {$g\!^{\smallvee}$};
\node at (1,1) {$g\!^{\smallvee}$};
\node at (1.5,0.5) {$g$};
\node at (0.5,0.5) {$g$};
\end{tikzpicture} &
\begin{tikzpicture}[scale=0.4]\tiny
\node at (1,0) {0};
\node at (1,1) {$0$};
\node at (1.5,0.5) {0};
\node at (0.5,0.5) {0};
\end{tikzpicture} & 
\begin{tikzpicture}[scale=0.4]\tiny
\node at (1,0) {$a\!^{\smallvee}$};
\node at (1,1) {$0$};
\node at (1.5,0.5) {$0$};
\node at (0.5,0.5) {$0$};
\end{tikzpicture} & $\cdots$ &
\begin{tikzpicture}[scale=0.4]\tiny
\node at (1,0) {$a\!^{\smallvee}$};
\node at (1,1) {$0$};
\node at (1.5,0.5) {0};
\node at (0.5,0.5) {$0$};
\end{tikzpicture} &
\begin{tikzpicture}[scale=0.4]\tiny
\node at (1,0) {$a\!^{\smallvee}\!\!-\!1$};
\node at (1,1) {$1$};
\node at (1.5,0.5) {0};
\node at (0.5,0.5) {0};
\end{tikzpicture}\\
\hline
\end{tabular}
\end{center}
The following lemma follows immediately from considering this table. 
\begin{lem} Suppose $\Sigma_{W,H}$ is a K3 orbifold with $W=x_0^p +f(x_1,x_2,x_3)$. Then $p-1 | 24.$
\end{lem}

Let $\tilde{X}$ be a crepant resolution of $\mf X=\Sigma(W,G)$, and $\tilde{X}^\vee$ a crepant resolution of the mirror. The fixed locus of $s$ is a disjoint union of curves an isolated fixed points. As before, let $f_1$ be the number of isolated fixed points, $N_1$ the number of curves, and $g_1$ the sum of the genera of the curves. 
\begin{cor}\label{cor:primeinv} Suppose $p>2$. Then
$N_1=g_1^\vee+1$ and
$$f_1+f_1^\vee+4=\frac{(p-2)}{(p-1)} 24.$$
\end{cor}
\begin{proof}
Using the table, it is easy to see
$$N_1=g^\vee+1, g_1=g.$$
Additionally,
$$f_1=(p-2)a^\vee-2.$$
Combining this with $(p-1)a+(p-1)a^\vee=24$, we obtain the statement in the theorem. 
\end{proof}

This corollary implies that Berglund--H\"ubsch mirror symmetry agrees with mirror symmetry for lattice polarised K3 surfaces.  We briefly recall the latter. 

Given a smooth K3 surface $\Sigma$, $\Lambda=H^2(\Sigma,\ZZ)$  is 
equipped with a lattice structure via the cup product taking values in $H^4(\Sigma;\ZZ)=\ZZ.$ Let $S_\Sigma:=\Lambda \cap H^{1,1}(\Sigma; \CC)$ be the \emph{Picard lattice} of $\Sigma$. 

Let $M$ be a hyperbolic lattice with signature $(1,t-1)$. A K3 surface $\Sigma$ is called \emph{$M$-polarized} if there exists a primitive embedding $M \hookrightarrow S_\Sigma$. Given a non-symplectic automorphism $s$ of $\Sigma$, the invariant sublattice $S(s):=\Lambda^s$ is in fact a primitive sublattice of the Picard lattice.  

\begin{defn} Given $M$ a primitive hyperbolic sublattice of $\Lambda=H^2(\Sigma,\ZZ)$ of rank at most 19 such that 
$$M^\perp=U \oplus M^\vee,$$
$M^\vee$ is defined to be the \emph{mirror} lattice to $M$. 
\end{defn}

Recall that we have restricted to the case where $s$ has prime order $p>2$ (we have discussed $p=2$ in \cite{CKV}). We now show that if two K3 surfaces with prime order automorphisms arise as crepant resolutions of a mirror pair of Berglund-H\"ubsch orbifolds, they have mirror lattices. In this case, $M:=S(s)$ is $p$-elementary. That is, $M^*/M =(\ZZ/p \ZZ)^{\oplus a}$, and it is completely classified by its rank $r$ and $a$. Then, by \cite{AST}, the fixed locus of $s$ is either just isolated points or a disjoint union of $N$ curves, of which $N-1$ are rational and the remaining one has genus $g$, and  $f$ isolated points. Set $m=\frac{22-r}{p-1}$. Moreover, \cite{AST} states (in a slightly different form) that for $p=3,5,7,13$, if the fixed locus contains a curve,
\begin{itemize}
\item $m=2g+a, -g+N=\frac{r-11+p}{p-1}.$
\end{itemize}
Notice that these are the only prime orders we need to consider, as we have shown that $p-1 | 24$. 

Lattice mirror symmetry exchanges $(r,a)$ with $(20-r,a)$.
\begin{thm}\label{thm:LPK3} Let $\Sigma_{W,H}$ and $\Sigma_{W^\vee,H^\vee}$ be mirror K3 orbifolds with prime order $p>2$ automorphisms $s, s^\vee$, and let $\Sigma$ and $\Sigma^\vee$ be crepant resolutions with automorphisms also denoted $s, s^\vee$. Then  $\Sigma$ and $\Sigma^\vee$ are mirror as lattice polarized K3 surfaces. 
\end{thm}
\begin{proof}
Corollary \ref{cor:primeinv} relates the invariants $(g,N,f)$ and $(g^\vee,N^\vee,f^\vee)$. It is enough to show that these relations give the mirror relations on $(r,a)$, namely that 
$$(r^\vee,a^\vee)=(20-r,a).$$ Notice that there is always a fixed curve when the K3 is a hypersurface in weighted projective space of this form. Therefore, we see that
\begin{equation*}
\begin{split}
r^\vee=(-g^\vee+N^\vee)(p-1)+(11-p)=(-N+g+2)(p-1)+(11-p)\\
=(p-1)(2-\frac{r-11+p}{p-1})+11-p=20-r.
\end{split}
\end{equation*}
Finally, this implies
$$a^\vee=2g^\vee-\frac{22-r^\vee}{p-1}=2(N-1)-\frac{2+r}{p-1}.$$
Using that $N=\frac{r-11+p}{p-1}+g,$ we obtain that
$$a^\vee=2g-\frac{22-r}{p-1}=a.$$
\end{proof}

\bibliographystyle{amsplain}
\bibliography{references}

\vspace{+16 pt}
\noindent A. Chiodo\\
\noindent Institut de Math\'ematiques de Jussieu -- Paris Rive Gauche \\
\noindent Sorbonne Universit\'e, UMR 7586 CNRS,\\
\noindent alessandro.chiodo@sorbonne-universite.fr 

\vspace{+16 pt}
\noindent E.~Kalashnikov \\
\noindent Department of Mathematics, FAS, Harvard University \\
\noindent One Oxford Street, Cambridge, MA 02138\\
\noindent kalashnikov@math.harvard.edu

\end{document}